\documentclass[11 pt]{amsart}
\usepackage{amsmath, amsthm, amssymb, amsfonts}
\usepackage{hyperref}
\usepackage{graphicx}
\usepackage{eucal}
\graphicspath{{Images/}}
\usepackage{float}
\usepackage{tikz}
\usepackage{caption}
\usepackage{subcaption}
\usepackage{tikz-cd}
\usepackage{mathtools}
\usepackage{xcolor}
\usepackage{stmaryrd,wasysym,bm}
\usepackage{accents}
\usepackage{amsmath}
\usepackage[all]{xy}
\usetikzlibrary{calc,angles,quotes}
\usepackage{tikz}
\usetikzlibrary{calc}
\usepackage{tikz}
\usetikzlibrary{arrows,chains,matrix,positioning,scopes}

\makeatletter
\tikzset{join/.code=\tikzset{after node path={%
			\ifx\tikzchainprevious\pgfutil@empty\else(\tikzchainprevious)%
			edge[every join]#1(\tikzchaincurrent)\fi}}}
\makeatother

\tikzset{>=stealth',every on chain/.append style={join},
	every join/.style={->}}
\title{Convergence-Symmetric Metric Spaces}
\author[Sa\u{g}lam]{\.{I}smail Sa\u{g}lam}
\address{\.{I}smail Sa\u{g}lam, Department of Aerospace Engineering, Adana Alparslan T\"urke\c{s} Science and Technology University, Adana, T\"urkiye and  \\
Max-Planck-Institut für Mathematik,
Vivatsgasse 7, 53111 Bonn, Germany
\\
}
\email{isaglamtrfr@gmail.com}

\author[Ohshika]{Ken'ichi Ohshika}
\address{Ken'ichi Ohshika, Department of Mathematics, Gakushuin University, Mejiro, Toshima-ku, Tokyo 171-8588, Japan and
Max-Planck-Institut für Mathematik,
Vivatsgasse 7, 53111 Bonn, Germany
}
\email{ohshika@math.gakushuin.ac.jp}

\author[Papadopoulos]{Athanase Papadopoulos}
\address{Athanase Papadopoulos, IRMA, CNRS and University of Strasbourg, Strasbourg, France and
 Max-Planck-Institut für Mathematik,
Vivatsgasse 7, 53111 Bonn, Germany 
}
\email{papadop@math.unistra.fr}

\author[Eyido\u{g}an]{Sad{\i}k Eyido\u{g}an}
\address{Sad{\i}k Eyido\u{g}an, Department of Mathematics, \c{C}ukurova University, Adana, T\"urkiye}
\email{seyidogan@cu.edu.tr}

\newtheorem{theorem}{Theorem}[section]
\newtheorem{corollary}[theorem]{Corollary}
\newtheorem{lemma}[theorem]{Lemma}
\newtheorem{proposition}[theorem]{Proposition}

\theoremstyle{definition}
\newtheorem{definition}[theorem]{Definition}
\theoremstyle{definition}
\newtheorem{example}[theorem]{Example}
\theoremstyle{remark}
\newtheorem{remark}[theorem]{Remark}
\theoremstyle{definition}
\newtheorem{question}[theorem]{Question}

\newcommand{\R}{{\mathbb R}}

\newcommand{\F}{{\mathcal F}}
\newcommand{\N}{{\mathbb N}}

\numberwithin{equation}{section}
\numberwithin{table}{section}

\begin{document}

\begin{abstract}
We study some basic properties of spaces which satisfy the usual axioms of a metric space but without the symmetry axiom. 
We realised that such a study is needed in view of the relatively recent appearance
of several papers on natural asymmetric metrics to which the available theories
do not apply. One prominent example is Thurston's metric  on the Teichmüller spaces of hyperbolic surfaces of finite type, introduced by Thurston in 1985, with several variants and generalisations.
 Another asymmetric metric is the earthquake metric, also introduced by Thurston and about which several basic questions remain open. Other asymmetric metrics we consider here include the Funk metric, the Apollonian metric and the left Hausdorff metric. We are particularly interested in questions of completeness and completion and in associated notions of boundary at infinity for these asymmetric metrics. In this paper, after discussing several examples, we prove results such as a Banach fixed point theorem, an Arzelà--Ascoli theorem and a Hopf--Rinow  theorem adapted to this asymmetric setting.  We  introduce a property we call ``convergence-symmetry" which turns out to be crucial in the study of metric completeness of some asymmetric metrics. This property is stronger than a property which was formulated by Herbert Busemann around 1970, and which we call the \emph{Busemann condition}.  Every convergence-symmetric metric space has a unique minimal completion
which is  convergence-symmetric. This does not hold for spaces satisfying Busemann's condition. Several examples we consider satisfy Busemann's condition but are not convergence-symmetric.  We have included throughout the paper a certain number of open questions.

\end{abstract}
\maketitle
\tableofcontents

%Apollonian semi-metric
%
%Hausdorff 

\textcolor{red}{}

\section{Introduction}\label{section:Intro}

In this paper, we shall deal with some variants of the notion of metric space. We first give some precise definitions. We start with the usual notion of a metric space as defined by Fréchet \cite{Fre} and Hausdorff \cite{Hau}, which we call, for our purpose here, ``symmetric metric space".

\begin{definition}[Symmetric metric space]
\label{symmetric-metric}
A symmetric metric space is a non-empty set $X$ together with a function (called symmetric distance function)
 $d: X\times X\to [0,\infty)$ 
satisfying the following three properties:

\begin{enumerate}
    \item 
$d(x,y)=0$ if and only if $x=y$,
\item 
$d(x,y)+d(y,z)\geq d(x,z)$ for all $x,y,z$ in $X$.
\item \label{S3}
$d(x,y)=d(y,x)$ for all $x,y \in X$.
\end{enumerate}
\end{definition}
We call such a function $d$ a \emph{symmetric distance function}.
Property \ref{S3} in this definition is the symmetry axiom. Examples of metrics which were studied extensively in the last few decades in geometry and low-dimensional topology  (in particular the Thurston metric), show that this axiom is too strong, and we think that it impedes many applications. In what follows, we shall review 
several important classes of distance functions which do not satisfy this axiom.
Nevertheless, as is well known, this axiom is useful for studying a large class of examples. Also, the theory of symmetric metric spaces itself is rich. For example, the symmetry axiom is used in the fact that every metric space can be embedded in a complete metric space. %and acquires there a natural boundary. 
One of our goals in this paper is to prove an analogous property for certain asymmetric metric spaces. For this, we study spaces satisfying a property   introduced in the paper \cite{SOP1}, which we call \emph{convergence-symmetry} and which we shall recall now. We start with the following notion  of metric space which is weaker than the one given in Definition \ref{symmetric-metric}:

\begin{definition}[Metric space]
\label{metric}
A  metric space is a non-empty set $X$ together with a function 
$$d: X\times X\to [0,\infty),$$ called a metric,
satisfying the following two properties:

\begin{enumerate}
    \item 
$d(x,y)=0$ if and only if $x=y$,
\item 
$d(x,y)+d(y,z)\geq d(x,z)$ for all $x,y,z$ in $X$.
\end{enumerate}
\end{definition}

We shall use the terminology ``asymmetric metric space", or ``non-symmetric metric space", for a set $X$ equipped with a function $d:X\times X\to [0,\infty)$ 
satisfying the axioms of Definition \ref{metric} and where the symmetry axiom fails, that is,  where there exist two points $x$
and $y$ in $X$ satisfying $d(x,y)\neq d(y,x)$. 

We shall study convergence of sequences and completions in metric spaces, and 
for this, the symmetry axiom will be replaced by the following much weaker one:
\begin{definition}[Convergence-symmetric metric space]\label{def:cs}
Let $X$ be a non-empty set and $d:X\times X \to [0,\infty)$ a metric.
 We say that $(X,d)$ is a convergence-symmetric metric space
if $d$ satisfies the following additional property, called the convergence-symmetry property:
for any two sequences $(p_n)$ and $(q_n)$ in $X$, we have
\begin{equation}\label{eq:distance}
d(p_n,q_n)\to 0 \ \text{if and only if}\ d(q_n,p_n)\to 0.
 \end{equation}   
\end{definition}

There are good examples of asymmetric metrics satisfying the convergence-symmetry property. We shall give several such examples which justify the general study of these spaces. Before that, let us first recall a property introduced by Herbert Busemann in his study of asymmetric metric spaces, which is an alternative axiom  different from  (\ref{eq:distance}).

Let $(X,d)$ be a metric space. It may happen that for a sequence $(x_n)$ in $X$ satisfying $d(x_n,x)\to 0$ , the sequence $d(x,x_n)$ does not converge to 0 as $n\to \infty$. This is a serious problem, for in this case one cannot define  from the metric $d$ a topology  on $X$ in the usual way;. To deal with this problem, Busemann introduced in \cite{Busemann1970} the following axiom, which we call \emph{Busemann's axiom}, or \emph{Busemann's condition}.
\begin{definition}[Metric space satisfying Busemann's axiom]
A metric space $(X,d)$ is said to satisfy Busemann's axiom if 
for any $x$ in $X$ and for any sequence $(x_n)$ in $X$, we have 
$$d(x_n,x)\to 0 \ \text{if and only if}\  d(x,x_n)\to 0.$$
\end{definition}
For a (not necessarily symmetric) metric, we can define two topologies compatible with the metric, the forward topology and the backward topology.
If $(X,d)$ is a metric space satisfying Busemann's axiom, then the two topologies  coincide, as we shall see in \S3.
They agree with the topology on $X$  induced by the symmetric metric
$$d_{\mathrm{arith}}(x,y)=\frac{1}{2}(d(x,y)+d(y,x)),$$
or, equivalently, by the symmetric metric
$$d_{\mathrm{max}}(x,y)=\mathrm{max}(d(x,y),d(y,x)).$$

 Busemann's axiom is weaker than the convergence-symmetry axiom introduced in Definition \ref{def:cs}.

We can talk about continuous functions between spaces  satisfying Busemann's axiom.
Furthermore, for such a space $X$, one can define in a natural way forward Cauchy sequences, forward completeness, backward Cauchy sequences and backward completeness, see \cite[Definition 6.3]{2012-Hilbert}. The space $X$ may be backward complete but not forward  complete, see \cite[Proposition 6.4]{2012-Hilbert} and Example  \ref{ex:Funk}  below. 
This last example also shows that there may be no forward/backward completion of a space which is forward/backward incomplete unless we allow the distance function to take the value $\infty$.
 Things become more natural from the point of view of completeness if $(X,d)$ satisfies the convergence-symmetry property: forward Cauchy sequences are backward Cauchy and vice versa; hence forward complete spaces are backward complete and vice versa. Thus, we have a good notion of completeness for convergence-symmetric spaces. Furthermore, there is a smallest complete metric space in which a convergence-symmetric space may be embedded isometrically and densely in a canonical way. These completeness properties were the initial motivation for the study of metric spaces satisfying the convergence-symmetry property.

The rest of this paper is organised as follows. 

In Section \ref{section:examples}, we provide a list of examples of asymmetric metric spaces which show our motivation for the study of such spaces. This includes spaces satisfying Busemann's axiom and others which do not satisfy it. Examples satisfying the convergence-symmetry property and others not satisfying it are included as well.

 In Section \ref{section:basic}, we introduce two topologies induced from an asymmetric metric, the forward topology and the backward topology. We show that the two topologies coincide if the metric satisfies Busemann's axiom.
 
In Section \ref{section:completion}, we study Cauchy sequences and we construct the 
completion of a convergence-symmetric metric space.
  We characterise compact metric spaces, and we show that sequential compactness is equivalent to topological compactness. 

Section \ref{section:weak} is concerned with convergence-symmetric Minkowski normed spaces. These are the linear versions of convergence-symmetric Finsler spaces, which we study in Section \ref{s:Finsler}.

In Section \ref{section:Ascoli}, we prove analogues of Arzel\`a--Ascoli's,  Baire category and Hopf--Rinow theorems for convergence-symmetric metric spaces.

In Section \ref{s:Thurston}, we study examples of Thurston-type weak metrics, that is, analogues in several variants of the metric he defined on the Teichm\"uller spaces of complete hyperbolic metrics on surfaces of finite type and we formulate several related open questions.

%
% ??? saying that the Banach Fixed Point Theorem can be generalized to complete convergence-symmetric metric spaces.
%
% In Section \ref{section:triangles}, we prove that the space of labeled Euclidean triangles introduced in Section \ref{section:examples} is complete. 
% 
% In Section \ref{Finsler-convergence-symmetric}, ???
%
% 
%In Section \ref{section:Baire} we show that a complete convergence-symmetric metric space is a Baire space. 
%

\section{Examples of asymmetric metric spaces}\label{section:examples}

\begin{example}
    The function $d:\mathbb{N}\times \mathbb{N} \longrightarrow \R $ defined by 
    $$ d(x,y) =
\begin{cases}
      y-x, & y \geq x\\
   x-y+1, & y<x
\end{cases}$$
 is an asymmetric metric. In addition, this metric  is convergence-symmetric. Indeed, the convergence \(d(p_n, q_n) \to 0\) implies that for sufficiently large \(n\), we have \(p_n = q_n\). Therefore, \(d(p_n, q_n) \to 0\) if and only if \(d(q_n, p_n) \to 0\). 

\end{example}

\begin{example}
   For a positive real number $\alpha$,  the function $d:\R\times \R\longrightarrow \R $ defined by 
$$ d(x,y) =
\begin{cases}
      x-y, & x \geq y\\
      \alpha(y-x), & y>x
\end{cases}$$
is an asymmetric metric space which is convergence-symmetric.

\end{example}

\begin{example}
    The function $d:\R\times \R\longrightarrow \R $ defined by 

    $$ d(x,y) =
\begin{cases}
      x-y, & x \geq y\\
   1, & y>x
\end{cases}$$
 is an asymmetric metric space. This metric does not satisfy Busemann's axiom.
\end{example}

\begin{example}\label{busemanamaconvergesdegil}
    The function $d:\R\times \R\longrightarrow \R $ defined by 
$$ d(x,y) =
\begin{cases}
      e^y - e^x, & y \geq x \\
   e^{-y} - e^{-x}, & y<x
\end{cases}$$
 is an asymmetric metric. We can easily see that this  metric  is not convergence-symmetric: Consider the sequences \( (p_n) := (n + \frac{1}{n}) \) and \( (q_n ) := (n) \). Then, while \(d(p_n, q_n) \to 0\), we have \(d(q_n, p_n) \to \infty\). 
 On the other hand, this metric satisfies Busemann's axiom.
\end{example}

\begin{example}\label{busemannbölümuzayı}
    The function $d:[0,\infty)\times [0,\infty )\longrightarrow \R $ defined by 
$$ d(x,y) =
\begin{cases}
      e^y - e^x, & y \geq x \\
       x - y, & x > y
\end{cases}$$

\noindent is an asymmetric metric.  
It is not difficult to see that it satisfies Busemann's axiom. Furthermore, the metric $d$ is not convergence-symmetric. To see this, consider the sequences $( p_n ) := ( n + \frac{1}{n} )$ and $(q_n ) := (n)$. Then, while $d(p_n, q_n) \to 0$ as $n \to \infty$, we have  $d(q_n, p_n) \to \infty$.

%1. \textit{Non-negativity and identity of indiscernibles:} It is straightforward to see that $d(x,y) \geq 0$ for all $x, y \in \R$, and $d(x,y) = 0$ if and only if $x = y$.

%2. \textit{Triangle inequality:} We need to show that for all $x, y, z \in \R$, the inequality $$d(x,z) \leq d(x,y) + d(y,z)$$ holds.
%\begin{itemize}
 %   \item  \(x < y < z\) :
  %   Here, \(d(x,z) = e^z - e^x\), \(d(x,y) = e^y - e^x\), and \(d(y,z) = e^z - e^y\).  
   %  Clearly, \(d(x,z) = (e^y - e^x) + (e^z - e^y) = d(x,y) + d(y,z)\).

   %\item   \(x < z < y\) :
    % Here, \(d(x,z) = e^z - e^x\), \(d(x,y) = e^y - e^x\), and \(d(y,z) = y - z\).  
     %Since \(e^z - e^x \leq e^y - e^x\) (as \(z < y\)) and \(y - z \geq 0\), we have \(d(x,z) \leq d(x,y) + d(y,z)\).

   %\item \(y < x < z\)  :
    % Here, \(d(x,z) = e^z - e^x\), \(d(x,y) = x - y\), and \(d(y,z) = e^z - e^y\).  
    % Since \(e^z - e^x \leq (x - y) + (e^z - e^y)\), the inequality holds.

   %\item  \(y < z < x\)  :
    % Here, \(d(x,z) = x - z\), \(d(x,y) = x - y\), and \(d(y,z) = e^z - e^y\).  
     %Clearly, \(x - z \leq (x - y) + (e^z - e^y)\), so the inequality holds.

  % \item \(z < x < y\)  :
   %  Here, \(d(x,z) = x - z\), \(d(x,y) = e^y - e^x\), and \(d(y,z) = y - z\).  
    % Since \(x - z \leq (e^y - e^x) + (y - z)\), the inequality is satisfied.

  %\item  \(z < y < x\)  :
   %  Here, \(d(x,z) = x - z\), \(d(x,y) = x - y\), and \(d(y,z) = y - z\).  
    % Clearly, \(x - z = (x - y) + (y - z) = d(x,y) + d(y,z)\).
%\end{itemize}
   
%In all cases, the triangle inequality is satisfied. Therefore, the function $d(x,y)$ satisfies all the conditions of a metric. \\

\end{example}

\begin{example}[The Funk metric]
\label{ex:Funk}
Let $\Omega$ be a closed bounded convex subset of $\R^n$ with non-empty interior $\mathring{\Omega}$. Given a pair of distinct points $x,y\in\mathring{\Omega}$, we denote the ray with origin $x$ and passing through  $y$ by $R(x,y)$. Let $a^+$ be the intersection point of $R(x,y)$ with the boundary $\partial \Omega$ of $\Omega$.  
Similarly, let $a^-$ be the intersection point of $R(y,x)$ with $\partial \Omega$. 
See Figure \ref{convexset}.  We recall that the \emph{Funk metric} on $\mathring{\Omega}$ is defined by the formula
$$\mathcal{F}(x,y)=
 \log\frac{\lvert x - a^+\rvert }{\lvert y-a^+\rvert},$$
 where $\vert \cdot \vert$ denotes the ordinary Euclidean norm of $\R^n$.
 \end{example}

   %   \begin{figure}
  %  \includegraphics{a-geodesic.pdf}
 %   \end{figure}

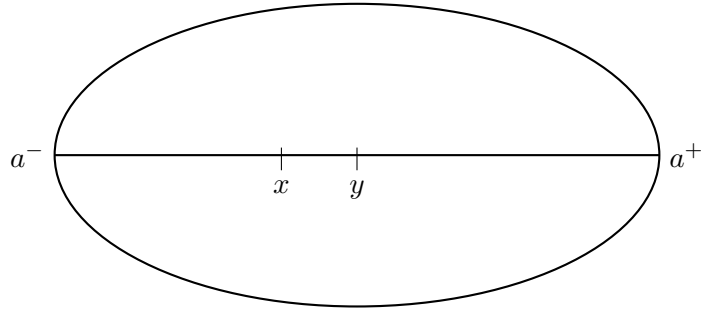
\begin{figure}[h]
\centering
 \begin{tikzpicture}
  % Çember
  \draw[thick] (0,0) ellipse (4cm and 2cm);

  % Yatay çizgi
  \draw[thick] (-4,0) -- (4,0);

  % Oklar ve etiketler
  \node[left] at (-4,0) {\(a^-\)};
  \node[right] at (4,0) {\(a^+\)};

  \draw[-] (-1,0.1) -- (-1,-0.2) node[below] {\(x\)};
  \draw[-] (0,0.1) -- (0,-0.2) node[below] {\(y\)};

\end{tikzpicture}
\caption{A closed bounded convex set in $\R^n$ with non-empty interior, and the definition of the Hilbert metric.}
\label{convexset}
\end{figure}
The Funk metric is an asymmetric metric. It is known that this metric is forward complete, but not backward complete (see Definition \ref{d:symmetry} for the definition, and 
see \cite[Propostion 6.4]{2012-Hilbert} for the proof). 
It follows from Lemma \ref{Cauchy} that the Funk metric is not convergence-symmetric. 

\begin{proposition}
The Funk metric satisfies Busemann’s axiom but not the convergence-symmetry property.
\end{proposition}

\begin{proof}
  Let $\Omega$ be a closed bounded convex set in $\mathbb{R}^n$ with a non-empty interior. Denote the Funk metric on $\mathring{\Omega}$ by $\mathcal{F}_\Omega$. 
Suppose that $(x_n)$ is a sequence in $\mathring{\Omega}$ and that $x$ a point in $\mathring{\Omega}$ such that $\mathcal{F}_\Omega(x_n,x)$ converges to $0$. From the definition of the Funk metric, we have (using the symbols $a_n^-, a_n^+$ to denote $a^-, a^+$ in Example \ref{ex:Funk}  replacing $y$  with $x_n$)
\[
\frac{|x_n - a_n^-|}{|x - a_n^-|} \to 1 \quad \text{as } n \to \infty.
\]
We aim to show that $\mathcal{F}_\Omega(x, x_n) \to 0$. By taking a subsequence, we can assume that 
  $$\frac{| x - a_{n_k}^+ | }{| x_{n_k} - a_{n_k}^+ |} \to \delta \neq 1 \quad \text{as } k \to \infty,$$  allowing $\delta$ to be $\infty$.
  Using the properties of limits and performing some arithmetic manipulations on the Funk distance, we derive the following system of equations:
  \begin{align*}
      1&= \lim_{k \to \infty} \frac{| x_{n_k} - a_{n_k}^- | }{| x - a_{n_k}^- |}=\lim \frac{| a_{n_k}^+ - a_{n_k}^- | - | x_{n_k} - a_{n_k}^+ | }{| x - a_{n_k}^- |}\\
      &= \lim_{k \to \infty} \left\lbrace \frac{| a_{n_k}^+ - a_{n_k}^- | }{| x - a_{n_k}^- |} - \frac{ | x_{n_k} - a_{n_k}^+ | }{| x - a_{n_k}^+ |}\cdot\frac{| x - a_{n_k}^+ |}{| x - a_{n_k}^- |} \right\rbrace \\
      & = \lim_{k \to \infty}  \frac{| a_{n_k}^+ - a_{n_k}^- |  - \frac{1}{\delta} \cdot | x - a_{n_k}^+ | }{| x - a_{n_k}^- |} =  \lim_{k \to \infty}  \frac{| a_{n_k}^+ - a_{n_k}^- |  - | x - a_{n_k}^+ | + \frac{\delta -1}{\delta} \cdot | x - a_{n_k}^+ | }{| x - a_{n_k}^- |}  \\
      &=\lim_{k \to \infty} \frac{| x - a_{n_k}^- | + \frac{\delta -1}{\delta} \cdot | x - a_{n_k}^+ | }{| x - a_{n_k}^- |} = 1 +  \frac{\delta -1}{\delta}  \cdot \lim_{k \to \infty} \frac{ | x - a_{n_k}^+ | }{| x - a_{n_k}^- |}, 
  \end{align*}
  where we define $1/\infty=0$ and $(\infty-1)/\infty=1$.
  There is a subsequence of  $(\frac{ | x - a_{n_k}^+ | }{| x - a_{n_k}^- |})$ whose limit is neither $0$ nor $\infty$ since the set $\Omega$ is a closed and bounded convex set with non-empty interior. This leads to a contradiction unless $\delta=1$. Thus, $\delta$ must be equal to $1$. As a result, $\mathcal{F}_\Omega(x,x_n)$ converges to $0$. 
A similar argument shows the reverse implication. Therefore, the Funk metric satisfies Busemann's axiom.

 %Recall that the Funk metric is not convergence-symmetric since it is forward complete but not backward complete. Using the definition, we can also show that the Funk metric does not satisfy the convergence-symmetry property. Let us find $\{x_n\}$ and $\{y_n\}$ in $\Omega$ such that while the sequence $\mathcal{F}_\Omega(y_n,x_n)$ converges to 0, $\mathcal{F}_\Omega(x_n,y_n)$ converges to a non-zero value. To see how to select $\{x_n\}$ and $\{y_n\}$, let us consider $\Omega := \{ (x,y) \in \mathbb{R}^2 \ : \  x^2 + y^2 < 1\} $.
 %Choose  $x_n$ and $y_n$ as $\left( 1- \frac{1}{n},0\right) $ and $\left( 1- \frac{1}{2n},0\right), $ respectively. In this case, while the sequence $\mathcal{F}_\Omega(y_n,x_n)$ converges to $0 $, sequence $\mathcal{F}_\Omega(x_n,y_n)$ converges to $\log 2$. Moreover, this method can be done for arbitrary ray in a closed and bounded convex set.

We noted that the Funk metric is not convergence-symmetric, as it is forward complete but not backward complete. We can also show  directly from its definition that this metric fails to satisfy the convergence-symmetry property. Specifically, we aim to find sequences $(x_n)$ and $(y_n)$ in $\Omega$ such that while $\mathcal{F}_\Omega(y_n, x_n)$ converges to $0$, the sequence $\mathcal{F}_\Omega(x_n, y_n)$ converges to a non-zero value.

We set $\Omega := \{ (x, y) \in \mathbb{R}^2 : x^2 + y^2 \leq  1\}$. Let $(x_n)$ and $(y_n)$ be the sequences $\left( 1 - \frac{1}{n}, 0 \right)$ and $\left( 1 - \frac{1}{2n}, 0 \right)$ respectively. Then, the sequence $\mathcal{F}_\Omega(y_n, x_n)$ converges to $0$, while the sequence $\mathcal{F}_\Omega(x_n, y_n)$ converges to $\log 2$.

This construction can be easily generalised to an arbitrary ray in the Funk topology of any closed and bounded convex set with non-empty interior.  
\end{proof}
 
\begin{example}[The Hilbert metric] \label{ex:Hilbert}
The Hilbert metric $H$ on $\mathring{\Omega}$ is defined by
$$H(x,y)=\frac{1}{2}[\mathcal{F}(x,y)+\mathcal{F}(y,x)]=\frac{1}{2}\log\frac{\lvert x - a^+\rvert }{\lvert y-a^+\rvert}+\frac{1}{2} \log\frac{\lvert y - a^-\rvert }{\lvert x-a^-\rvert},$$
 and it is a complete symmetric metric. 
  \end{example}

 \begin{example}[The weighted Funk metric]
 \label{ex:weighted-Funk}
 
 For every $t\in(0,1)$,  the weighted Funk metric $\mathcal{F}^a_t$ on $\mathring{\Omega}$ is defined by 
$$\mathcal{F}^a_t(x,y)=(1-t)\mathcal{F}(x,y)+t\mathcal{F}(y,x).$$
This metric is asymmetric if $t\neq 1/2$. The weighted Funk metric $\mathcal{F}^a_{1/2}$ is the Hilbert metric. It is not difficult to see that $\mathcal{F}^a_t$ is convergence-symmetric, using the fact that if $(\F^a(x_n, y_n))$ converges to $0$ then both $(\F(x_n, y_n))$ and $(\F(y_n,x_n))$ converge to $0$.

Another family of weighted Funk metrics $\mathcal{F}^m$ is defined by 
$$\mathcal{F}^m_t(x,y)=\max\{(1-t)\mathcal{F}(x,y),t\mathcal{F}(y,x)\}$$
for all $t\in (0,1)$.
The metric $\mathcal{F}^m_t$ is asymmetric for each $t\neq 1/2$ and again it is easy to see that it is convergence-symmetric, as in the case of the metric $\mathcal F^a_t$.
 \end{example}

%
%Let $\Omega$ be, as before, a closed bounded convex subset of $\R^n$ with non-empty interior $\mathring{\Omega}$  and let $t\in (0,1)$. From the Funk metric $\mathcal{F}(x,y)$, 

In some instances, it is useful to relax the first axiom of a metric space and to introduce the following notion.

\begin{definition}[Weak metric]
    A weak metric space is a non-empty set  $X$ together with a function 
$$d: X\times X\to [0,\infty)$$
satisfying the following properties:
\begin{enumerate}
    \item 
$d(x,y)=0$ if  $x=y$,
\item 
$d(x,y)+d(y,z)\geq d(x,z)$ for all $x,y,z$ in $X$.
\end{enumerate}
\end{definition}

The following terminology for metric spaces applies to weak metric spaces as well:

\begin{definition}[Busemann's axiom and convergence symmetry]
A weak metric satisfies Busemann's axiom if $d(x_n,x)\to 0$ if and only $d(x,x_n)\to0$ as $n\to \infty$  for all sequences  $(x_n)$ in $X$ and $x\in X$. A weak metric is called convergence-symmetric if  for any two sequences $(x_n)$ and $(y_n)$ in $X$, $d(x_n,y_n)\to 0$ implies $d(y_n,x_n)\to 0$ as $n\to \infty$. 

\end{definition}

\begin{remark}[The metric space associated with a weak metric space satisfying Busemann's axiom]
    Let $(X,d)$ be a weak metric space satisfying Busemann's axiom. We say that two elements $x$ and $x'$ in $X$ are equivalent and write $x\sim x'$ if $d(x,x')=0$. Let us verify that $\sim$ is an equivalence relation. Assume that $x\sim y$, that is,  $d(x,y)=0$. Let $(x_n)$ be the constant 
    sequence such that $x_n=x$ for all $n$. Then we have
    $0=d(x,y)=\lim_{n\to\infty}d(x_n,y)=\lim_{n\to \infty}d(y,x_n)=d(y,x).$
 It follows that $d(y,x)=0$ and $y\sim x$. Now assume that $x\sim y$ and $y\sim z$. Then $d(x,y)+d(y,z)=0$. The triangle inequality implies that $d(x,z)=0$. Hence $x\sim z$.
Thus we have shown that $\sim$ is an equivalence relation.

    Let $\tilde{X}$ be the set of equivalence classes of elements in $X$, and denote the equivalence class of an element $x$ by $[x]$. We define $\tilde{d}([x],[y]])=d(x,y).$
Let us sshow that this is well defined, that is, $\tilde{d}([x],[y])$ does not depend on the representatives $x$ and $y$.  Assume that $[x]=[x']$ and $[y]=[y']$. We have  $d(x',y')=d(x,x')+d(x',y')+d(y',y)\geq d(x,y),$therefore $d(x',y')\geq d(x,y)$. Similarly, $d(x',y')\leq d(x,y)$. Thus $d(x,y)=d(x',y')$ and this proves the claim. It is easy to see that $\tilde{d}$ satisfies Busemann's axiom.
In the same way, we can see that if $d$ is convergence-symmetric, then $\tilde{d}$ is convergence-symmetric.

 \end{remark}

For a metric space $(X,d)$ and  an equivalence relation $\sim$ on $X$, we can define a weak metric $d'$ on $X/\sim$ as in Example \ref{ex:equivalence} below.
%A metric $d'$ on the quotient space $X/\sim$ can be defined using sequences of points within the equivalence classes. 
Even if we start with a metric $d$ satisfying Busemann's axiom, the metric $d'$ does not necessarily satisfy this property. This will be seen through a counterexample; see Examples \ref{ex:equivalence} and \ref{rem:equivalence}.
%In order to build this counterexample, we shall use the metric defined next.

\begin{example}\label{ex:equivalence}
    Let $(X,d)$ be a  metric space and let $\sim$ be an equivalence relation on $X$. 
    Let us denote the equivalence class of an element $x\in X$ by $[x]$. 
    We can define a weak metric on $X/\sim$ as follows:
$$d'([x],[y])=\inf \{d(x_1,y_1)+d(x_2,y_2)+\cdots + d(x_{n-1},y_{n-1})+d(x_n,y_n)\},$$
   where the infimum is taken over all finite sequences $(x_1,x_2,\dots, x_n)$ and $(y_1,y_2,\dots,y_n)$ such that $[x_1]=[x]$, $[y_n]=[y]$, $[y_i]=[x_{i+1}]$ for $i=1,2,\dots n-1$.
\end{example}

\begin{example} \label{rem:equivalence}
Let \(d\) be a metric satisfying Busemann's axiom and \(\sim\) an equivalence relation on \(X\). We aim to show that the weak metric \(d'\), defined by \(d\) and \(\sim\) as in Example \ref{ex:equivalence}, does not necessarily satisfy Busemann's axiom. Consider the metric \(d : [0,\infty) \times [0,\infty) \to \mathbb{R}\), defined in Example \ref{busemannbölümuzayı}:
\[
d(x, y) =
\begin{cases}
e^y - e^x, & y \geq x, \\
x - y, & x > y,
\end{cases}
\]
and  an equivalence relation \(\sim\) on $[0,\infty)$ such that all natural numbers are identified and  any other real number is equivalent only to itself. 
We denote the equivalence class of an element in $x\in \R$ by $[x]$. Let us show that the weak metric \(d'\), derived from \(d\) and \(\sim\) does not satisfy Busemann's axiom.

 Consider a sequence \(\alpha_n =[x_n] = \left[n + \frac{1}{n}\right]\) and the point \(\left[1\right]\). We first compute \(d'\left(\alpha_n, \left[1\right]\right)\). By definition of \(d'\), we have:
\[
d'\left(\alpha_n, \left[1\right]\right) \leq d\left(n + \frac{1}{n}, n\right) = n + \frac{1}{n} - n.
\]
which implies that \(d'\left(\alpha_n, \left[1\right]\right) \to 0\).

We next compute \(d'\left(\left[1\right], \alpha_n\right)\). 
Suppose, seeking a contradiction, that \(d'\left(\left[1\right], \alpha_n\right) \to 0\). 
By the definition of convergence and \(d'\), this implies that for any \(\epsilon > 0\), there exists a finite sequence of points \((x_1, \dots, x_m)\) and \((y_1, \dots, y_m)\) such that:
\[
[x_1] = [1], \quad [y_m] = [\alpha_n], \quad [y_i] = [x_{i+1}] (i=1, \dots, m-1),
\]
and
\[
\sum_{i=1}^m d(x_i, y_i) < \epsilon.
\]
Using the triangle inequality and the definition of the equivalence relation, we obtain
\[
d\left(x_1, n + \frac{1}{n}\right) \leq \sum_{i=1}^m d(x_i, y_i) < \epsilon.
\]
This leads to a contradiction, as \(\lim_{n\to \infty} \inf_{m\in \mathbb{N}}d(m, n + \frac{1}{n}) \neq 0\). Hence, \(d'\left(\left[1\right], \alpha_n\right)\) does not converge to \(0\).

Thus, the weak metric \(d'\) does not satisfy Busemann's axiom, for   \(d'([x_n], [x]) \to 0\) whereas  $d'([x], [x_n])$ does not converge to $0$.
\end{example}

Before proceeding, we give another example of weak metric.

\begin{example}[The Apollonian weak metric of the upper-half-plane]
    Consider the upper-half plane $\mathbb{H}^2=\{\zeta \in \mathbb{C}: \text{Im} (\zeta)>0 \}$. Let 
   $M: \mathbb{H}^2\times \mathbb{H}^2\to \R$
be a function defined by 
    $$M(\zeta,\zeta')=\sup_{x\in \R}\frac{\lvert\zeta'-x\rvert}{\lvert\zeta-x\rvert}.$$

\noindent     Then 
    $$\delta(\zeta,\zeta'):=\log M(\zeta,\zeta')$$
    \noindent is a weak metric. It  is neither symmetric nor a metric, that is, it does not separate points. It turns out that this weak metric is an analogue of the Thurston metric for the case of the Teichm\"uller space of the  torus, see \cite{2005c, OMP2020}. This metric is a special case of  metrics which were called in \cite{2007ab} the Apollonian weak metrics.
    % that we recall in the next example.

\end{example}
%
%\begin{example}[Apollonian weak metric]
%\label{s:Apollonian-metric}
%
%We recall the notion of Apollonian weak metric defined in  \cite{2007ab} (called there the Apollonian semi-metric).
%Let $A$ be an open subset of $\mathbb{R}^n$ and let $\partial A= \overline{A}\setminus A$ be its boundary. We assume as in \cite{2007ab} that either $A$ is bounded or $\partial A$ is unbounded. 
% The \emph{Apollonian weak metric $\delta_{A}$ on $A$} is defined by the formula
% \[\delta_A(x,y)=\sup_{a\in \partial A} \log \frac{\vert x-a\vert}{\vert y-a\vert}.
% \]
%  
%\end{example}
% 
% \begin{question}
% I will include a question about the convergence-symmetry of the Apollonian metric but I have to think about it before ???
% \end{question}
% 
% 
%\begin{example}
%    Let $A$ be a proper open subset of $\R^n$.  Consider  the function $i_A$ on $A \times A$ defined by
%
%    $$i_A(x,y)=\log(1+\frac{\lvert x-y\rvert}{d(x,\partial A)}),$$
%where $\partial A$ is the boundary of $A$ and $d(x,\partial A)$ is the distance between $\partial A$ and $x$. The function $i_A$ is a weak metric on $A$. See \cite[Proposition 4.1]{2007ab}.
%%\end{example}
%
%Let $M$ be a differentiable manifold and $TM$ the tangent bundle of $M$. For any $x\in M$, we denote the tangent space at $x$ by $T_xM$.
%
% \begin{question}
% Same thing about the above metric, I have to figure out also if it has a name ???
% \end{question}

\section{On the topology of the metrics satisfying Busemann's axiom}\label{section:basic}

We first introduce two topologies for a (not necessary symmetric) metric space.
Before that, we define the term forward/backward convergence.

\begin{definition}[Forward/backward convergence]
Let $(X,d)$ be a metric space.
We say that a sequence $(x_n)$ in $X$ forward converges to $x\in X$ if $d(x_n,x)\to 0$ as $n\to\infty$, and backward converges to $x$ if $d(x,x_n) \to 0$ as $n \to \infty$.
\end{definition}

Now we define forward and backward topologies.
\begin{definition}
\label{topology}
Let $(X,d)$ be a metric space.
A subset $U$ of $X$ is said to be  forward open  when for any $x \in U$ and any sequence $(x_n \in X)$ forward converging to $x$, there exists $n_0$ such that $x_n$ lies in $U$ for every $n \geq n_0$.
In the same way, $U$ is said to be  backward open when for any $x \in U$ and a sequence $(x_n \in X)$ backward converging to $x$, there exists $n_0$ such that $x_n$ lies in $U$ for every $n \geq n_0$.
It is quite easy to see that both collections of forward open sets and backward open sets satisfy the axioms of open sets.
We call the topology on $X$ defined by the forward open sets the forward topology, and the one defined by the backward open sets the backward topology.
\end{definition}

The following is an easy consequence of Busemann's axiom.

\begin{lemma}
\label{Busemann same top}
If $(X,d)$ satisfies Busemann's axiom, then any sequence forward converging to $x$ also backward converges to $x$.
Therefore, the forward topology and the backward topology on $X$ coincide.
\end{lemma}
\begin{proof}
The first statement is just a consequence of the definition of forward/backward convergence.

We have only to show that the open sets coincide for both topologies.
Let $U$ be a forward open set, and $x$ a point in $U$.
Suppose that a sequence  $(x_n)$ in $X$ backward converges to $x$.
Then the first part implies that $(x_n)$ also forward converges to $x$, and since $U$ is forward open, there is $n_0$ such that $x_n$ lies in $U$ for every $n \geq n_0$.
This shows that $U$ is also backward open.
The same argument shows also the opposite implication.
\end{proof}
Therefore, if $(X,d)$ satisfies Busemann's axiom, we can simply talk about the topology of $X$ and convergence in $(X,d)$.
This means that we can also talk about other notions used for topological spaces, such as closed set, closure, interior, density, isolated points, accumulation points etc. without specifying forward or backward.

%\begin{definition}[Metric]
%Let $X$ be a non-empty set.  A   metric on $X$ is a function $d: X\times X\to [0,\infty)$ satisfying the following properties:
%
%\begin{enumerate}
%    \item 
%$d(x,y)=0$ if and only if $x=y$,
%\item 
%$d(x,y)+d(y,z)\geq d(x,z)$ for all $x,y,z$ in $X$.
%
%
%\end{enumerate}
%A set together with a metric is called a metric space. We say that the metric $d$ satisfies  Busemann's axiom if the following property holds:
%
%$d(x_n,x)\to 0$ as $n\to \infty$ if and only if $d(x,x_n)\to 0$ as $n\to \infty$ for all sequence $(x_n)$ in  $X$ and for all $x\in X$.
%
%\end{definition}

We can define two kinds of symmetrisation of a metric as follows.

\begin{definition}[Arithmetic and sup symmetrisation]
Let $(X,d)$ be a metric space. We define  the \emph{arithmetic symmetrisation of $d$} by  
$$d_{\mathrm{arith}}(x,y)=\frac{d(x,y)+d(y,x)}{2}$$
and the \emph{sup symmetrisation of $d$} by  
$$d_{\mathrm{max}}(x,y)=\max\{d(x,y), d(y,x)\}.$$
\end{definition}

Observe that both $d_{\mathrm{max}}$ and $d_{\mathrm{arith}}$ are symmetric metrics on $X$, and also that the following inequality is valid for every $x,y \in X$:
$$\frac{d_{\mathrm{max}}(x,y)}{2}\leq d_{\mathrm{arith}}(x,y)\leq d_{\mathrm{max}}(x,y).$$
This implies that $d_{\mathrm{max}}$ and $d_{\mathrm{arith}}$ are equivalent symmetric metrics.

 The following is an easy consequence of Lemma \ref{Busemann same top}.
\begin{proposition}\label{prop:converge}
Let $(X,d)$ be a metric space satisfying Busemann's axiom.
Then for any $x\in X$ and $(x_n)$ a sequence in $X$, the following are equivalent:
\begin{enumerate}
    \item 
    $(x_n)$ converges to $x$ in $(X,d)$.
    \item 
    $(x_n)$ converges to $x$  in $(X, d_{\max})$.
    \item 
   $(x_n)$ converges to $x$  in $(X, d_{\mathrm{arith}})$.
\end{enumerate}
\end{proposition}

As a corollary, we have the following.

\begin{corollary}
\label{same topology}
For a metric space $(X,d)$ satisfying Busemann's axiom, the topologies of $(X,d), (X,d_{\max})$ and $(X, d_{\mathrm{arith}})$ are the same.
\end{corollary}

\noindent Let $(X,d)$ be a metric space satisfying Busemann's axiom. 
We denote by $\mathcal{T}_d$ the topology induced by $d$, or equivalently by 
$d_{\max}$ or $d_{\mathrm{arith}}$.

\begin{definition}[Sum, Euclidean sum and product of metric spaces]\label{rem:product}
Let $(X_1, d_1),\dots, (X_n,d_n)$ be  metric spaces. Then the following three functions are metrics on $\prod_{i=1}^nX_i$:
\begin{enumerate}
    \item 
    $$\mu_1: (a,b)\mapsto \sum_{i=1}^{n}d_i(a_i,b_i).$$
    \item 
    $$\mu_2: (a,b)\mapsto \sqrt{\sum_{i=1}^n(d_i(a_i,b_i))^2}.$$
    \item 
    $$\mu_{\infty}: (a,b)\mapsto \max_i\{d_i(a_i,b_i)\}.$$
\end{enumerate}
\end{definition}

For each $a,b \in \prod_iX_i$, the following inequalities holds:
\begin{align}
	\label{equation:product}
\mu_{\infty}(a,b)\leq \mu_2(a,b)\leq \mu_1(a,b)\leq  n \mu_{\infty}(a,b).
\end{align}

Observe also that $\mu_{1,\mathrm{arith}}(a,b)=\sum_{i=1}^nd_{i,\mathrm{arith}}(a_i,b_i)$. It is easy to see that if each $d_i$ satisfies Busemann's axiom, then $\mu_1,\mu_2$ and $\mu_{\infty}$ satisfy it as well and that if each $d_i$
is convergence-symmetric, then $\mu_1,\mu_2$ and $\mu_{\infty}$ are convergence-symmetric as well.

Let $(X,d)$ be a metric space and let $Y\subset X$. The restriction of $d$ to $Y\times Y$ is a metric. We can see the following immediately from the definitions.

\begin{proposition}
 If $d$ satisfies Busemann's axiom, then its restriction to $Y$ satisfies this axiom as well. If $(X,d)$ is convergence-symmetric, then $(Y,d) $ is also convergence-symmetric.
\end{proposition}

\begin{definition}[Diameter]\label{def:diameter}
Let $(X,d)$ be a metric space. The diameter of a subset $A\subset X$ is the quantity
$$\mathrm{diam}(A)=\sup\{d(r,s): r,s \in A\}.$$
\end{definition}

%
%\begin{definition}[Distance from a point to a set]
%    Let $(X,d)$ be a metric space, $x$ a point in $X$ and $A$ a subset of $X$. The  distance from $x$ to $A$ is the quantity
%    $$dist(x,A)=\inf\{d(x,a): a\in A\}.$$
% The  distance from $A$ to $x$ is the quantity
%        $$dist(A,x)=\inf\{d(a,x): a\in A\}.$$
%
%\end{definition}
%
%
%
%
%
%
%
%
%\
%

For a sequence $(x_n)$ satisfying Proposition \ref{prop:converge}, we let $x=\lim x_n$. Observe that such a point $x$ is unique.  We call it the limit of $(x_n)$. We have the following:

\begin{proposition}
 If $(X_i,d_i)$ with $i=1,\dots, n$ is a sequence of metric spaces satisfying Busemann's axiom and  if $(x_n)$ a sequence in $\prod X_i$, then the following are equivalent: 
\begin{enumerate}
    \item 
    $(x_n)$ is convergent with respect to $\mu_{\infty}$.
    \item 
    $(x_n ) $ is convergent with respect to $\mu_2$
    \item 
        $( x_n ) $ is convergent with respect to $\mu_1$.
        \item 
            $( x_n ) $ is convergent with respect to $\mu_{1,\mathrm{arith}}$.

\end{enumerate}

\begin{proof}
	The equivalence of (1), (2) and (3) follows form Inequality \ref{equation:product}. Properties (3) and (4) in the statement are equivalent since $\mu_1$ satisfies Busemann's axiom.
\end{proof}
\end{proposition}

\medskip

If $Y$ is a subset of $X$, then we can talk about relative open and closed subsets of $Y$. Note that for a subset $A$ of $X$, the closure of $A$, the interior of $A$, the exterior of $A$ and the boundary of $A$ coincide with the closure, interior, exterior and boundary respectively of $A$ when $X$ is regarded as a topological space.

%\begin{definition}[Dense subset]
%We call a subset $A$ of $X$ dense in $X$ if for each $x\in X$ there is a sequence $(a_n)$ in $A$ such that $d(a_n,x)\to 0$. 
%\end{definition}
%Note that since $d(a_n,a)\to 0$ exactly when $d_{\mathrm{arith}}(a,a_n)\to  0$, it follows that $A$ is dense in $X$ if and only if $A$ is dense in $X$ in the topological sense, that is, if $Cl(A)=X$. It follows that the following statements are equivalent:
%
%\begin{enumerate}
%    \item 
%    For every $x\in X$, $dist(x,A)=0$;
%    \item 
%       for every $x\in X$, $dist(A,x)=0$;
%        \item 
%        $A$ has a non-empty intersection with every open set of $X$.
%
%\end{enumerate}
%

\begin{remark}
    Let $(X_i,d_i)$, $1\leq i\leq n$ be a finite family of metric spaces satisfying Busemann's axiom. Consider the metric $\mu_1$ on $\prod X_i$. Since $\mu_{1,\mathrm{arith}}(a,b)=\sum_{i=1}^nd_{i,\mathrm{arith}}(a_i,b_i)$, it follows that 
    the family of sets of the form $U_1\times\dots \times U_n$, where $U_i$ is open in $X_i$, form a basis for the 
    topology $\mathcal{T}_{\mu_1}$. 
\end{remark}

Let $(X_i,d_i)$ be a metric space satisfying Busemann's axiom  for every $1\leq i\leq n$. Let $P=\prod X_i$, let $\pi_i:P\to X_i$ be the projection onto $X_i$ and let $\mu_1$, $\mu_2$, $\mu_{\infty}$ be the metrics defined before.

\begin{proposition}
With the above notation, we have the  following:

\[
\mu_{\infty,\mathrm{arith}} \leq \mu_{2,\mathrm{arith}}
\leq \mu_{1,\mathrm{arith}} \leq n\mu_{\infty,\mathrm{arith}},
\]
\[
\mu_{\infty,\max} \leq \mu_{2,\max}
\leq \mu_{1,\max} \leq n\mu_{\infty,\max}.
\]
\begin{proof}
	This follows from Inequality \ref{equation:product}.
\end{proof}
\end{proposition}
It follows that all the above symmetric metrics generate the same topology on $P$. We call this topology the \emph{product topology}. Now we see that $\pi_i$ is continuous if we consider the product topology on $P$ and the topology $\mathcal{T}_{d_i}$ on $X_i$. 

\begin{proposition}
For $1\leq i\leq n$, let $(X_i,d_i)$ be a metric space satisfying Busemann's axiom.
    Suppose that $(x_n)$ is a sequence in $P$. Then 
    $(x_n)$ converges in $P$ if and only if $(\pi_i(x_n))$ converges in $X_i$ for each $i$. If this occurs, then 
    $\pi_i(\lim x_n)=\lim \pi_i(x_n)$.

    \begin{proof}
    	A sequence $(x_n)$ is convergent in $P$ if and only if it is convergent
    	with respect to the metric $\mu_{1,\mathrm{arith}}$. This is equivalent to saying
    	that, for each $i$, the sequence $(\pi_i(x_n))$ is convergent with respect to
    	$d_{i,\mathrm{arith}}$. Since each $d_i$ satisfies Busemann's axiom, convergence
    	with respect to $d_{i,\mathrm{arith}}$ is equivalent to convergence with respect to
    	$d_i$. Hence, $(x_n)$ converges in $P$ if and only if $(\pi_i(x_n))$ converges
    	in $X_i$ for each $i$.
    	
    	The last assertion follows from the fact that a sequence $(y_n)$ converges to
    	$y$ in $X_i$ with respect to $d_i$ if and only if it converges to $y$ with
    	respect to $d_{i,\mathrm{arith}}$.
    \end{proof}
\end{proposition}
%
%\begin{remark}
%Let $(X,d)$ be a metric space satisfying Busemann's axiom and $(x_n)$ be a sequence in $X$.
%    If $(x_n)$ converges to $x\in X$, then any subsequence of $(x_n)$ converges to $x$.
% 
%\end{remark}
\section{Cauchy Sequences and completion}\label{section:completion}
 
Let $(X,d)$ be a metric space satisfying Busemann's axiom.

\begin{definition}[Forward/backward Cauchy sequences and forward/backward completeness]\label{d:symmetry}
	 A sequence $(x_i)$ in $X$ is called a forward Cauchy sequence if for every $\epsilon >0$ there exists an integer $N$ such that $d(x_i,x_{i+k})<\epsilon$ for every $i\geq N$ and $k\geq 0$. The sequence $(x_i)$ is called a backward Cauchy sequence if for every $\epsilon >0$ there exists an integer $N$ such that $d(x_{i+k},x_i)<\epsilon$ for every $i\geq N$ and $k\geq 0$. The space $X$ is called forward complete if for every forward Cauchy sequence $(x_n)$, there exists $x\in X$ such that $d(x_n,x)\to 0$ as $n\to \infty$. Similarly, $X$ is called backward complete if for every backward Cauchy sequence $(x_n)$, there exists $x\in X$ such that $d(x,x_n)\to 0$ as $n\to \infty$.
\end{definition}

\begin{example}[A space satisfying Busemann's axiom with no backward completion]
This example may be regarded as a continuation of Example \ref{ex:Funk};
it concerns the Funk metric. Let $\Omega$ be a closed bounded convex subset of
$\mathbb{R}^n$ with non-empty interior. Let $\bar{x}$ and $\bar{y}$ be two
distinct points on the boundary of $\Omega$. Let $L_{\bar{x}}$ and
$L_{\bar{y}}$ be two line segments contained in $\Omega$ whose relative
interiors are contained in $\mathring{\Omega}$ and whose endpoints $\bar{x}$ and $\bar{y}$ are in $\partial \Omega$.  

Let $(x_n)$ be a sequence in $L_{\bar{x}}\cap \mathring{\Omega}$ converging
to $\bar{x}$ with respect to the Euclidean distance, and let $(y_n)$ be a
sequence in $L_{\bar{y}}\cap \mathring{\Omega}$ converging to $\bar{y}$ with
respect to the same distance. As before, we denote the Funk metric on
$\mathring{\Omega}$ by $\mathcal{F}$. Then it is not difficult to see that
the sequences $(x_n)$ and $(y_n)$ are backward Cauchy with respect to
$\mathcal{F}$. Moreover,
\[
\mathcal{F}(x_n,y_n)\to \infty
\quad \text{as } n\to\infty .
\]

It follows from the preceding discussion that the Funk metric cannot be
extended to a metric $d$ defined on a larger backward complete space $X$ satisfying the following
property:
 If $(a_n)$ is a sequence in $\mathring{\Omega}$ and $a$ is a point in $X$
such that
$
d(a,a_n)\to 0
\quad \text{as } n\to\infty,
$
and if $(b_n)$ is a sequence in $\mathring{\Omega}$ and $b$ is a point in $X$ such
that
$
d(b,b_n)\to 0
\quad \text{as } n\to\infty,
$
then
$
d(a,b)=\lim_{n\to\infty} d(a_n,b_n).
$
\end{example}

We note that if we allow $d$ to take the value $\infty$ in a larger space, then it is possible to construct a backward complete space containing $X$ as a dense subset.
See Algom--Kfir \cite{Al}.

\medskip

We shall see  in Theorem \ref{th:minimal completion} that every convergence-symmetric metric space has a unique minimal completion that is convergence-symmetric.

The following lemma shows that the notions of forward and backward Cauchy sequences coincide on a convergence-symmetric metric space.
\begin{lemma}
\label{Cauchy}
Let $(X,d)$ be  a convergence-symmetric metric space.
Then a sequence in $X$ is forward  Cauchy if and only if it is backward Cauchy.
\end{lemma}

	\begin{proof}
	Let $(x_n)$ be a forward Cauchy sequence in $X$. 
	Assume that $(x_n)$ is not backward Cauchy. 
	Then we can find $\epsilon>0$ such that    for all $N>0$ there exist integers $m(N)\geq n(N)\geq N$ such that $d(x_{m(N)},x_{n(N))})\geq\epsilon$. Now since $(x_n)$ is forward Cauchy, we have $d(x_{n(N)},x_{m(N)})\to 0$ as $N\to \infty$. 
	From the convergence-symmetry property we have $d(x_{m(N)},x_{n(N))})\to 0$ as well. 
	This is a contradiction. Therefore every forward Cauchy sequence is backward Cauchy. By the same proof, the converse holds.  
	\end{proof}

Thus, in a convergence-symmetric   metric space $(X,d)$, we can simply call a sequence Cauchy if it is forward (or equivalently backward) Cauchy. Let $(x_n)$ be a Cauchy sequence in $X$. Then for each $\epsilon >0$ there exists $N$ such that for all $n,m \geq N$ $d(x_n,x_m)<\epsilon$.

\begin{definition}[Complete convergence-symmetric  metric space]
	A convergence-symmetric  metric space  is said to be complete if every Cauchy sequence in this space converges.
\end{definition}

 \begin{remark}
 Every convergent sequence in a convergence-symmetric metric space  is a Cauchy sequence.
 \end{remark}

\begin{theorem} \label{th:minimal completion} Every convergence-symmetric metric space has a unique minimal completion that is convergence-symmetric. 
\end{theorem}

\begin{proof}
Let $(X,d)$ be a convergence-symmetric metric space. We  construct a completion of $X$ with respect to $d$. The outline is the same as the one for symmetric metric spaces, but we point out the steps which depend crucially on the fact that our space satisfies the convergence-symmetry property.

(1)
 Consider the set of Cauchy sequences in $X$. 
We say that two Cauchy sequences $(p_n)$ and $(q_n)$ are equivalent if
\[
d(p_n,q_n)\to 0 \quad \text{as } n\to\infty.
\]
This defines an equivalence relation on the set of Cauchy sequences in \(X\).
Indeed, reflexivity is immediate. Symmetry follows from the convergence-symmetry
property. Finally, if \(d(p_n,q_n)\to 0\) and \(d(q_n,r_n)\to 0\), then the triangle
inequality gives
\[
d(p_n,r_n)\leq d(p_n,q_n)+d(q_n,r_n)\to 0,
\]
which proves transitivity.
 
(2)
If \((p_n)\) and \((q_n)\) are Cauchy sequences, then \((d(p_n,q_n))\) is
Cauchy in \(\mathbb R\). Indeed, by the triangle inequality,
\[
d(p_n,q_n)\leq d(p_n,p_m)+d(p_m,q_m)+d(q_m,q_n),
\]
and the analogous inequality with \(m\) and \(n\) interchanged also holds.
It follows that
\[
|d(p_n,q_n)-d(p_m,q_m)|\to 0
\]
as \(m,n\to\infty\). Hence \((d(p_n,q_n))\) converges.

(3)
Let $X^*$ be the set of equivalence classes of Cauchy sequences in $X$. If $P\in X^*$ and $Q\in X^*$ denote the equivalence classes of $(p_n)$ and $(q_n)$, then we define 
$$\Delta(P,Q)=\lim_{n\to \infty }d(p_n,q_n).$$
This limit exists since $(p_n)$ and $(q_n)$ are both  Cauchy.
Also if $(p'_n)$
and $(q'_n)$ are other representatives of $P$ and $Q$, then we have
$$d(p'_n,q'_n)\leq d(p'_n,p_n)+d(p_n,q_n)+d(q_n,q'_n),$$
  therefore the two sequences $(d(p'_n,q'_n))$ and $(d(p_n,q_n))$ have the same limit. It follows that $\Delta$ is well defined. 

(4)
It is easy to see that $\Delta$ is a function satisfying the  axioms of a metric space.

(5)
We show that $(X^*, \Delta)$ has the convergence-symmetry property. 

Let $(P_n)$ and $(Q_n)$ be two sequences in $X^*$ such that $\Delta(P_n,Q_n)\to 0$ as $n\to\infty$. 
Let $(p_{n,m})_m$ and $(q_{n,m})_m$ be representatives for $P_n$ and $Q_n$ respectively. Then we have 
$\lim_{n\to\infty}\lim_{m\to\infty}d(p_{n,m},q_{n,m})=0,$
which means that for every $\epsilon >0$, there exist $N$ and $M(n)$ such that if $n\geq N$ and $m \geq M(n)$, then $d(p_{n,m}, q_{n,m}) < \epsilon$.
This is equivalent to saying that there exist $n(k)$ and $m(k)$ for $k \in \N$ such that for every $i(k)\geq n(k)$ and $j(k) \geq m(k)$, we have $\lim_{k \to \infty}d(p_{i(k), j(k)}, q_{i(k), j(k)}) \to 0$.
Since $X$ has the convergence-symmetry property, this implies that $\lim_{k \to \infty} d(q_{i(k), j(k)}, p_{i(k), j(k)})=0$ for every $i(k) \geq n(k)$ and $j(k) \geq m(k)$.
By the c. above equivalence, this implies that $\lim_{n\to\infty}\lim_{m\to\infty}d(q_{n,m},p_{n,m})=0,$ which means that $\Delta(Q_n, P_n) \to 0$.
%Assume that $\Delta(Q_n,P_n)\neq 0$. Then, passing to a subsequence if necessary, we may assume that $\Delta(Q_n,P_n)\to A$, where $A\in (0,\infty]$. Then
%we have 
%$$\lim_{n\to\infty}\lim_{m\to\infty}d(q_{n,m},p_{n,m})=A,$$
%which implies that there are sequences $(p_{n(k),m(k)})_k$ and $(q_{n(k),m(k)})_k$ such that 
%$$\lim_{k\to \infty}d(q_{n(k),m(k)},p_{(n(k),m(k))})\to A.$$
%It follows from the convergence-symmetry of $X$ that 
%$(d(p_{n(k),m(k)},q_{(n(k),m(k))}))$ does not converge to $0$.
%Again, passing to a subsequence if necessary, we see that this contradicts the fact that 
%$$\lim_{n\to\infty}\lim_{m\to\infty}d(p_{n,m},q_{n,m})=0.$$

(6)
We prove that $X^*$ is complete. 
Since we already showed that $X^*$ is convergence-symmetric, we have only to prove that $X^*$ is forward complete.
Let $(P_n)$ be a Cauchy sequence in $X^*$, and $(p_{n,m})_m$ a sequence which is a representative of $P_n$.
Since $(P_n)$ is forward Cauchy, for any $k \in \N$, there is $N(k)$ such that for any $n \geq N(k)$ and any non-negative integer $i$, we have $\Delta(P_{n+i}, P_n) < 1/k$.
We can take $N(k)$ so that it is monotone increasing with respect to $k$.
This means that for any $n\geq N(k)$, there is $M(n,i)$ such that for any $m \geq M(n,i)$, we have $d(p_{n+i, m}, p_{n, m}) < 1/k$.
On the other hand, since each $(p_{n,m})$ is a forward Cauchy sequence, there is $M(n,k)$ such that for any $m \geq M(n,k)$ and any non-negative integer $i$, we have $d(p_{n, m+i}, p_{n,m}) < 1/k$.
Again we can take $M(n,k)$ so that it is monotone increasing with respect to both $n$ and $k$.

Now we consider the sequence $Q=(q_k):=(p_{N(k),M(N(k), k)})_k$.
We can see that this sequence is forward Cauchy.
Indeed for any non-negative integer $i$, we have 
$$ d(p_{N(k+i), M(N(k+i), k+i)}, p_{N(k), M(N(k), k)})  $$
$$\leq d(p_{N(k+i), M(N(k+i), k+i)}, p_{N(k), M(N(k+i), k+i)}) $$
$$  +d(p_{N(k), M(N(k+i), k+i)}, p_{N(k), M(N(k),k)})  < 2/k.$$
  
We next show that $d(P_n, Q) \to 0$.
For any $k$, we can choose $K$ such that if $n > K$, then $n>N(k)$.
Then we have $\Delta(P_n, P_{N(k)}) < 1/k$, which means that there is $L(k)$ such that for any $m \geq L(k)$, we have $d(p_{n,m}, p_{N(k), m}) < 1/k$.
We assume that $L(k)\geq M(N(k), k)$, replacing $L(k)$ with a bigger one if necessary.
Then we also have $d(p_{N(k), m}, p_{N(k), M(N(k), k)})<1/k$, and hence $d(p_{n,m}, q_k)<2/k$.
Since $Q$ is also backward Cauchy, for sufficiently large $k$, we can take a larger $m$ so that $d(q_k, q_m)< 1/k$, and hence we have $\lim_{m\to \infty} d(p_{n,m}, q_m) \leq 3/k$.
Thus we have shown that $\Delta(P_n, Q) \to 0$.
%If $(P_n)$ is a Cauchy sequence in  $X^*$, let $(p_{n,m})_m$ be a representative for $P_n$. Consider the sequence $P=(p_{n,n})$. 
%Then  we have $P_n\to P$  as $n\to\infty$. 

(7)
For each $p\in X$, consider the constant sequence all of whose terms are $p$. This is evidently a Cauchy sequence. Let $P_p$ be the element of $X^*$ which contains this sequence. It is obvious that $\Delta(P_p,P_q)=d(p,q).$
 Therefore there is a distance-preserving map 
$\phi: X\to X^*$
given by $\phi(p)=P_p$.

(8)
It is not difficult to see $\phi(X)$ is dense in $X^*$. Indeed let $P\in X^*$ be represented by the sequence $(p_n)$. Let $P_n=Pp_n$. Then $P_n\to P$ as $n \to \infty$.

(9)
Let $(X',d')$ be a  metric space satisfying the convergence-symmetry property and assume that it is complete.  
Suppose that  there is a distance-preserving map $\phi':X\to X'$.
Then $\phi'$ can be extended to a map $X^*\to X'$ by sending 
an element $P$ represented by a Cauchy sequence $(p_n)$ to the unique limit  of $(\phi'(p_n))$ in $X'$.
This means that $X^*$ is the minimal complete convergence-symmetric metric space containing $X$.
 
\end{proof}

\begin{remark}
Let $(X,d)$ be a convergence-symmetric metric space. Let $(x_n)$ be a sequence in $X$. The following are equivalent:

\begin{enumerate}
    \item 
    $(x_n)$ is a Cauchy sequence with respect to $d$.
    \item 
    $(x_n)$ is a Cauchy sequence with respect to $d_{\mathrm{arith}}$.
    \item 
       $(x_n)$ is a Cauchy sequence with respect to $d_{\mathrm{max}}$.
\end{enumerate}

It follows from the above properties that the arithmetic (resp. max) symmetrisation of the completion of $(X,d)$ is the usual completion of the arithmetic (resp.  max) symmetrisation of the metric $d$.

\end{remark}

 The following theorem establishes a relation between the completeness of a metric space satisfying the convergence-symmetry property and the completeness in the usual sense of a symmetric metric space. The proof is a straightforward application of  Lemma \ref{Cauchy}.
\begin{theorem} \label{thm-crucial}
	Assume that $(X,d)$ is a convergence-symmetric metric space. Then the following are equivalent.
	\begin{enumerate}
		\item 
		$(X,d)$ is complete.
		\item 
		$(X,d_{\mathrm{arith}})$ is complete.
		\item 
		$(X,d_{\mathrm{max}})$ is complete.
		
	\end{enumerate}
%
%    \begin{proof}
%    Assume that $(X,d)$ is complete. Let $(x_n)$ be a Cauchy sequence in $X$ with respect to the metric $d_{\mathrm{arith}}$. Then $(x_n)$ is Cauchy with respect to $d$. It follows that $(x_n)$ converges to $x$ with respect to the metric $d$, thus $x_n$ converges to $x$ with respect to the metric $d_{\mathrm{arith}}$. Hence $(X,d_{\mathrm{arith}})$ is complete. Now assume that $(X,d_{\mathrm{arith}})$ is complete. Let $(x_n)$ be a Cauchy sequence with respect to the metric $d$. Then $(x_n)$ converges to $x$ with respect to $d_{artih}$ for some $x\in X$. Then $(x_n)$ converges to $x$ with respect to $d$. Therefore $(X,d)$ is complete. Since $d_{\mathrm{arith}}$ and $d_{\mathrm{max}}$ are equivalent metrics, the statements (2) and (3) are equivalent.
%    \end{proof}
\end{theorem}

Note that no equivalent statement holds for spaces satisfying Busemann's axiom.
The Funk metric is forward complete and not backward complete, and its arithmetic symmetrisation (which is the Hilbert metric) is complete.

\begin{definition}[Equivalence of convergence-symmetric metrics]
Two con\-vergence-symmetric metrics $d$ and $d'$ on a set $X$ are said to be equivalent if there exist constants $C,D>0$ such that the following inequalities hold for all $x,y\in X$:
$$Cd'(x,y)\leq d(x,y)\leq D d'(x,y).$$
\end{definition}
\noindent Note that equivalent metrics $d$ and $d'$ define the same topology, and that a sequence $(x_n)$ converges to $x$ with respect to $d$ if and only if it converges to $x$ with respect to $d'$. Also if $(x_n)$ is Cauchy with respect to $d$ then it is Cauchy with respect to $d'$, which implies that $(X,d)$ is complete if and only if $(X,d')$ is complete.

\begin{example}[Completeness of the weighted Funk metric]
Let $\Omega$ be a closed bounded convex subset of $\R^n$ with non-empty interior $\mathring{\Omega}$. 
Assume that $\partial\Omega$ does not contain any line segment. In this case, it is well known that the line segments are the only geodesics of the Hilbert metric $H$ on $\Omega$. 
We see that  $\mathring{\Omega}$ is geodesically complete then.  Therefore, the Hopf--Rinow theorem implies that $(\mathring{\Omega},H)$ is complete. By Theorem \ref{thm-crucial}, it follows that for any $t\in (0,1)$, the metric space $(\mathring{\Omega},\mathcal{F}^a_t)$ is complete. It also follows from the same proposition that
$(\mathring{\Omega},\mathcal{F}^m_t)$ is complete if and only if $\mathring{\Omega}$ is complete
with respect to the metric $\max\{\mathcal{F}(x,y),\mathcal{F}(y,x)\}$.
 
 \end{example}
 
\begin{definition}[Bounded set and bounded map]
\label{def:bounded-sets}
Let $(Y,d)$ be a metric space.
A subset $A$ of $Y$ is called bounded if its diameter is finite. Let $X$ be a set and  $f: X\to Y$  a map. We say that $f$ is bounded if its range is bounded. The set of bounded maps $X\to Y$ is denoted by $B(X,Y)$.
\end{definition}

\begin{remark}
   It is easy to see that a Cauchy sequence in a convergence-symmetric metric space is bounded as a subset.
\end{remark}

\begin{proposition}
\label{uniform symmetric}
Let $(Y,d)$ be a metric space.
Consider the following function $s:B(X,Y)\times B(X,Y)\to \R$:
$$s(f,g)=\sup \{d(f(x),g(x)): x\in X\}.$$
Then $s$ is a convergence-symmetric metric on $B(X,Y)$ provided that $(Y,d)$ is a convergence-symmetric metric space. 
    \begin{proof}
    It is easy to see that $s$ satisfies the first two axioms of a convergence-symmetric metric space.
    We show that the convergence-symmetry axiom holds as well. Let $(f_n)$ and $(g_n)$ be sequences  of bounded functions such that $s(f_n,g_n)\to 0$  as $n\to \infty$. 
  Assume, seeking a contradiction, that $s(g_n,f_n)$ does not converge to $0$. Then there is a sequence of elements $(x_n)$ in $X$ such that  $d(g_n(x_n),f_n(x_n))$ does not converge to $0$. 
    It follows that $d(f_n(x_n),g_n(x_n))$ does not converge to $0$ by the convergence-symmetry of $Y$. 
    This contradicts the assumption that 
    $s(f_n,g_n)\to 0$ as $n \to \infty$.
    \end{proof}
\end{proposition}

Next, we shall consider pointwise convergence and a sequential criterion for continuity.

Let $X$ be a set and $(Y,d)$ a convergence-symmetric  metric space. We saw in Proposition \ref{uniform symmetric} that the space $B(X,Y)$ of bounded functions from $X$ to $Y$  is endowed with a metric $s$, which we call the supremum metric.
The following is an immediate consequence of  the definition of $s$.
\begin{proposition}
    Let $(f_n)$ be a sequence in $B(X,Y)$ converging to $f\in B(X,Y)$. Then, for every $z\in X$,  $f_n(z)$ converges to $f(z)$.
\end{proposition}

Recall that given a metric space $(X,d)$ satisfying Busemann's axiom, we defined a topology on $X$, and denoted it by $\mathcal{T}_d$ (see \S \ref{section:basic}). 
Therefore we can talk about the continuity of functions between metric spaces satisfying Busemann's axiom.
The following three results are immediate from Definition \ref{topology}.

%\begin{definition}[Continuity]
% If $(X,d)$  and $(Y,e)$ are two metric spaces satisfying Busemann's axiom and 
%$f\colon X\to Y$ is a function, we say that $f$ is continuous at $x\in X$ if $f$ is continuous at $x$ with respect to the topologies $\mathcal{T}_d$  and $\mathcal{T}_e$ . We say that $f$ is continuous if it is continuous at every point of $X$. \end{definition}
%%If $Z$
%is any topological space, when we mention the continuity of a function $X\to Z$ or $Z\to X$, the topology $\mathcal{T}_d$ should be recalled.
%
\begin{proposition}
	\label{prop:seq-cont}
    Let $(X,d)$ and $(Y,e)$ be metric spaces satisfying Busemann's axiom and $f \colon X\to Y$ a function. Then $f$ is continuous at $x$ if and only if for each sequence $(x_n)$ in $X$ with $\lim_{n\to \infty} x_n =x$, we have $\lim_{n\to \infty} f(x_n)=f(x)$.
%    \begin{proof}
%    Assume that $f$ is continuous at $x$. Consider $d_{\mathrm{arith}}$ and $e_{\mathrm{arith}}$. Let $\lim x_n=x$. It follows that $x_n$ converges to $x$ with respect to $d_{\mathrm{arith}}$. Therefore $f(x_n)$ converges to $f(x)$ with respect to $e_{\mathrm{arith}}$. It follows that $f(x_n)$ converges to $x$ with respect to $e$.  Conversely assume that $f$ is not continuous at $x$. Then there is sequence $(x_n)$ that converges to $x$ with respect to $d_{\mathrm{arith}}$ and $f(x_n)$ does not converge to $f(x)$ with respect to $e_{\mathrm{arith}}$. It follow that $(x_n)$ converges to $x$ with respect to $d$ but $(f(x_n))$ does not converge to $f(x)$ with respect to $e$. 
%    \end{proof}
\end{proposition}

\begin{corollary}
    Let $(X,d)$ and $(Y,e)$ be two metric spaces satisfying Busemann's axiom and $f\colon X\to Y$ a function. Then $f$ is continuous if and only if for all $x\in X$ and for all $(x_n)$ with $\lim_{n\to \infty} x_n =x$, $\lim_{n\to\infty} f(x_n)$ exists and is equal to $f(x)$.
    \end{corollary}

\begin{corollary}
	\label{cor:metric-cont}
	Let $(X,d)$ be a metric space satisfying Busemann's axiom. Consider the product space $X\times X$ and the metric $\mu_1$ on $X\times X$ (See Definition \ref{rem:product}). Then the function $d: X\times X \to \R$ is continuous.
%	\begin{proof}
%		This follows immediately from Proposition \ref{prop:seq-cont}.
%	\end{proof} 
\end{corollary}

\begin{definition}[Compactness for a space satisfying Busemann's axiom]
Let $(X,d)$ be a metric space satisfying Busemann's axiom. We say that $X$ is compact if it is compact as a topological space with respect to the topology of $\mathcal{T}_d$ (see Corollary \ref{same topology}).
\end{definition}

Now we turn to compactness.

\begin{proposition}
\label{prop:compactness}
Let $(X,d)$ be a metric space satisfying Busemann's axiom.
    The following four conditions are equivalent.
\begin{enumerate}
    \item 
    $(X,d)$ is compact.
    \item 
    $(X,d_{\mathrm{arith}})$ is compact.
    \item 
    $(X,d_{\mathrm{max}})$ is compact.
    \item 
Every sequence in $X$ has a convergent subsequence with respect to the metric $d$.
\end{enumerate}
\end{proposition}
    \begin{proof}
    The equivalence of the first three conditions follows from the fact that the topology $\mathcal{T}_d$ coincides with the topology induced by the metric $d_{\mathrm{arith}}$ or $d_{\mathrm{max}}$, as shown in Corollary \ref{same topology}. A sequence in $X$ has a convergent subsequence with respect to $d$ if and only if it has a convergent subsequence with respect to $d_{\mathrm{arith}}$. This observation implies that the last condition is equivalent to the second one. This completes the proof of the proposition.
    \end{proof}

\begin{lemma}[Lebesgue number lemma]
    Let $(X,d)$ be a compact metric space satisfying Busemann's axiom. Let $\frak{A}$ be an open covering of $X$. Then there is a number $\delta>0$ such that for each subset of $X$ having diameter less than $\delta$, there exists an element of $\frak{A}$ containing it.

    \begin{proof} (See Definition \ref{def:diameter} for the diameter of a set.)
        Consider the metric $d_{\mathrm{max}}$. Since $(X,d_{\mathrm{max}})$ is compact, the classical Lebesgue Number Lemma implies that there exists $\delta$ such that if $diam_{d_{\mathrm{max}}}(A)<\delta$, then 
        $A$ is contained in an element of $\frak{A}$. Now the result follows from the observation that $diam_{d}(A)=diam_{d_{\mathrm{max}}}(A)$.
    \end{proof}

\end{lemma}

The following definition, which is classical in the setting symmetric metric spaces, will be also useful in our more general setting
\begin{definition} [Uniform continuity]
Let $(X,d_1)$ and $(Y,d_2)$ be metric spaces. A function $f: (X,d_1)\to (Y,d_2)$ is said to be uniformly continuous if for all $\epsilon >0$ , there exists $\delta >0$ such that $d_1(x,y)<\delta$ implies $d_2(f(x),f(y))<\epsilon$.
\end{definition}

It is not difficult to see that $f$ is uniformly continuous if and only if for any sequences $(x_n)$ and $(y_n)$ in $X$, $d_1(x_n,y_n)\to0$ implies that $d_2(f(x_n),f(y_n))\to 0$ as $n\to \infty$. Using this remark, we have the following.

\begin{lemma}
If $f: (X,d_1)\to (Y,d_2)$ is uniformly continuous, then $f: (X,d_{1,\mathrm{arith}})\to (Y,d_{2,\mathrm{arith}})$ is also uniformly continuous.
\begin{proof}
Assume that $d_{1,\mathrm{arith}}(x_n,y_n)\to 0$. Then,
$$d_1(x_n,y_n)\to 0 \ \text{and}\ d_{1}(y_n,x_n)\to 0.$$
Thus, by assumption, we have,
$$d_2(f(x_n),f(y_n))\to 0 \ \text{and}\ d_2(f(y_n),f(x_n))\to 0.$$

\noindent This clearly implies that $d_{2,\mathrm{arith}}(f(x_n),f(y_n))\to 0$.  
\end{proof}
\end{lemma}
%
%\begin{remark}
%Let $(X,d_1)$ and $(Y,d_2)$ be metric spaces satisfying Busemann's axiom. If $f: (X,d_1)\to (Y,d_2)$ is uniformly continuous, then  we showed that $f: (X,d_{1,\mathrm{arith}})\to (Y,d_{2,\mathrm{arith}})$ is uniformly continuous, hence $f$
%is continuous with respect to the two metrics $d_{1,\mathrm{arith}}$ and $d_{2,\mathrm{arith}}$. It follows that $f: (X,d_1)\to (Y,d_2)$ is continuous. 
%\end{remark}

\begin{lemma}
\label{lemma:uniform-continuity}
Let $(X,d_1)$ be a convergence-symmetric metric space and let $(Y,d_2)$ be an arbitrary metric space. If $f: (X,d_{1,\mathrm{arith}})\to (Y,d_{2,\mathrm{arith}})$ is uniformly continuous, then $f: (X,d_1)\to (Y,d_2)$ is uniformly continuous.
\begin{proof}
If $d_1(x_n,y_n)\to 0$, then $d_1(y_n,x_n)\to 0$ as well. It follows that $d_{1,\mathrm{arith}}(x_n,y_n)\to 0$. By assumption, $d_{2,\mathrm{arith}}(f(x_n),f(y_n))\to 0$. Finally, this implies that $d_2(f(x_n),f(y_n))\to 0$.
\end{proof} 
\end{lemma}

\begin{theorem}
\label{thm:uniform-continuity}
Let $(X,d_1)$ be a compact convergence-symmetric metric space and let $(Y,d_2)$ be a metric space satisfying Busemann's axiom. If $f: (X,d_1)\to (Y,d_2)$ is continuous, then $f$ is uniformly continuous with respect to the two metrics 
$d_1$ and $d_2$.
\begin{proof}
We start by observing that $f:(X,d_{1,\mathrm{arith}})\to (Y,d_{2,\mathrm{arith}})$ is continuous. Since $(X,d_{1,\mathrm{arith}})$ is compact, this implies that $f$ is uniformly continuous with respect to the two metrics $d_{1,\mathrm{arith}}$  and $d_{2,\mathrm{arith}}$.  Then Lemma 
\ref{lemma:uniform-continuity} directly implies the desired result.
\end{proof}

\end{theorem}

\begin{definition}[Standard associated 
bounded metric]
Let $(X,d)$ be a metric space. We call the metric   $\bar{d}(x,y)=\min\{d(x,y),1\}$ on $X$ the standard 
bounded metric associated 
with $d$. 
\end{definition} 
Observe that the standard associated bounded metric is a metric. The metric space $(X,d)$ is convergence-symmetric if and only if $(X,\bar{d})$ is convergence-symmetric. Also, under the assumption of convergence-symmetry, $(X,\bar{d})$ is complete if and only if $(X,d)$ is complete.

\begin{definition}
    Let $(Y,d)$ be a metric space and $J$  an index set. Consider the standard associated bounded metric $\bar{d}$ on $Y$. If $x=(x_{\alpha})$ and $y=(y_{\alpha})$ are points in $Y^{J}$, then we define 
    $$\bar{\rho}(x,y)=\sup\{\bar{d}(x_{\alpha},y_{\alpha}): \alpha \in J\}.$$
It is easy to check that $\bar{\rho}$ is a metric on $Y^J$. It is 
called the uniform metric on $Y^J$.   
\end{definition}

\begin{proposition}
Let $(Y,d)$ be a convergence-symmetric metric space. If the space $Y$ is complete with respect to $d$, then the space $Y^J$ is complete with respect to $\bar{\rho}$.
\end{proposition}
\begin{proof}
Let $(f_n)$ be a Cauchy sequence in $Y^J$. 
It follows that for every $\alpha \in J$, $(f_n(\alpha))$ is a Cauchy sequence in $Y$ with respect to $\bar{d}$. 
Hence it converges to a point $y_{\alpha}\in Y$.
Let $f\colon J\to Y$ be the function defined by $f(\alpha)=y_{\alpha}$. 
We claim that $(f_n)$ converges to $f$ with respect to the metric $\bar{\rho}$.

Give $\epsilon >0$, choose $N$ large enough so that $n\geq m \geq N$ implies  $\bar{\rho}(f_n,f_m)<\epsilon /2$. Then in particular 
$\bar{d}(f_n(\alpha),f_m(\alpha))<\epsilon/2$
 for all $n\geq m\geq N$ and for all $\alpha \in J$.
Since $(f_n(\alpha))$ converges to $f(\alpha)$, making $n\to \infty$, we see that 
$\bar{d}(f_m(\alpha),f(\alpha))\leq \epsilon/2$.
This inequality holds for all $m\geq N$ and for all $\alpha \in J$. Thus the claim is proved.
\end{proof}

 \begin{definition}[uniform convergence in spaces having the convergence-symmetry property]
Let $X$ be a set and $(Y,d)$ a metric space having the convergence-symmetry property. Let $(f_n)$ be a sequence of functions from $X$ to $Y$ and $f: X\to Y$ a function. Then, by Proposition \ref{uniform symmetric}, the following two conditions are equivalent.
\begin{enumerate}
    \item 
For every $\epsilon >0$, there exists a positive integer $N$ such that $d(f_n(x),f(x))< \epsilon$ for all $n\geq N$ and for all $x\in X$.
    \item 
    For every $\epsilon >0$, there exists a positive integer $M$ such that 
    $d(f(x),f_n(x))< \epsilon$ for all $n\geq M$ and for all $x\in X$.
    \end{enumerate}
If one condition holds (and hence, if both hold), then we say that $(f_n)$ converges uniformly to $f$. 
\end{definition}

\begin{remark}
If $(f_n)$ converges to $f$ uniformly with respect to the metric $d$, then $(f_n)$  also converges to $f$ uniformly with respect to the metrics $d_{\mathrm{arith}}$ and $d_{\mathrm{max}}$. Conversely, if $(f_n)$ converges uniformly to $f$ with respect to the metric $d_{\mathrm{arith}}$ (or $d_{\mathrm{max}}$), then $f_n$ converges to $f$ uniformly with respect to the metric $d$.
\end{remark}

\begin{proposition}
	\label{prop:uniform-convergence}
    Let $(f_n)$ be a sequence of continuous functions from a topological space $X$  to a metric space $Y$ having the convergence-symmetry property. If $(f_n)$ converges to $f$ uniformly, then $f$ is continuous.
    \end{proposition}
    \begin{proof}
        By the  remark above, one can replace $d$ with  $d_{\mathrm{arith}}$, which is a symmetric metric. Then the result follows from the uniform limit theorem.
    \end{proof}

\begin{remark}
    A closed subset $A$ of a complete convergence-symmmetric  metric space $(X,d)$ is also complete. Indeed, by Corollary \ref{same topology}, $A$ is closed with respect to the metric $d_{\mathrm{arith}}$ by. Therefore, it is complete with respect to the restriction of $d_{\mathrm{arith}}$. Hence,  by Theorem \ref{thm-crucial},  it is complete with respect to $d$. 
\end{remark}
\begin{proposition}
    Let $X$ be a topological space and  $(Y,d)$ a metric space with the convergence-symmetry property. The set of continuous functions $\mathcal{C}(X,Y)$ is closed in $Y^X$ under the uniform metric, and so is the set $B(X,Y)$ of bounded functions. Therefore, if $Y$ is complete and convergence-symmetric, then these spaces are complete with respect to the uniform metric.
\end{proposition}
\begin{proof}
The first part of the proposition is just a restatement of Proposition \ref{prop:uniform-convergence}. The rest of the proposition can be proved by replacing the uniform metric with its arithmetic or max symmetrisation.
\end{proof}

We end this section by an analogue of the Banach fixed point theorem for complete convergence-symmetric  metric space.

\begin{definition}
	Let \((X,d)\) be a metric space. A map \(T:X\to X\) is said to be a contraction
	if there exists a real number \(q\in[0,1)\) such that for all \(x,y\in X\),
	\[
	d(T(x),T(y))\leq qd(x,y).
	\]
\end{definition}

Observe that if $T$ is a contraction with respect to a metric $d$, then it is also a contraction with respect to the associated metric $d_{\mathrm{arith}}$. In particular, this implies that $T$ is continuous.

\begin{theorem}
If $(X,d)$ is a complete convergence-symmetric  metric space and $T\colon X\to X$ a contraction, then $T$ has a unique fixed point $x^*$. Furthermore, $x^*$ can be found as follows: start with any point $x_0 \in X$, and for $n\geq 1$ set $x_n=T(x_{n-1})$. Then $\lim_{n \to \infty} x_n=x^*$.
\end{theorem}
\begin{proof}
Consider the metric $d_{\mathrm{arith}}$. Since $T$ is a contraction with respect to $d_{\mathrm{arith}}$ and since $X$
is complete with respect to $d_{\mathrm{arith}}$, the classical Banach fixed point theorem implies that 
$T$ has a unique fixed point $x^*$, and that for any point  $x_0$  in $X$, we have $d_{\mathrm{arith}}(x_n,x^*)\to 0$.  Thus $d(x_n,x^*)\to 0$ as $n\to \infty$ and the result follows.
\end{proof}

\section{Convergence-symmetric Minkowski norms on $\R^n$} \label{section:weak}

We start with recalling the notion of Minkowski norm.

\begin{definition}[Minkowski norm]
    A Minkowski norm on a real vector space $V$ is a function $\Vert \cdot \Vert: V\to \R$ with the following properties:
    \begin{enumerate}
        \item 
        $\Vert x \Vert \geq 0$
for all $x \in V$, and $\Vert x \Vert=0$ if and only if $x=0$.  
\item 

$\Vert x +y\Vert\leq \Vert x \Vert+\Vert y\Vert$ for all $x,y\in V$.
\item 
$\Vert \alpha x \Vert=\alpha \Vert x \Vert$ for all $x\in V$ and for all $\alpha\geq 0$.
        \end{enumerate}
\end{definition}
  Note that a Minkowski norm is a norm if and only if $\Vert -x \Vert=\Vert x\Vert$ for all $x \in V$.  Now let us define $d \colon V \times V \to \R$ by  $d(x,y)=\Vert y-x\Vert$. 
  Then 
 $d$ is a metric on $V$, and it is symmetric if and only if $\Vert \cdot \Vert$ is a norm. 
 This metric is convergence-symmetric if and only if for any sequence $(x_n)$ such that $\lvert\lvert x_n \rvert\rvert \to 0$ as $n\to \infty$, we have $\lvert\lvert -x_n\rvert\rvert\to 0$
  as $n\to\infty$.

\begin{example}
We give an example of a Minkowski norm whose associated metric is not convergence-symmetric.
    Let $\mathbb{N}$ be the set of natural numbers and let $V\subset \R^{\mathbb{N}}$ be the set of functions $\mathbb{N}\to \R$  such that for each $f\in V$ there exists $N>0$ with the property that 
    $f(n)=0$ for all $n>N$.
    
    For each $x\in \mathbb{R}$, we set
\begin{equation*}
\lvert\lvert  x \rvert \rvert_n =\begin{cases}
          x \quad &\text{if} \, x  \geq 0\\
          -nx \quad &\text{if} \,  x <0.\\
     \end{cases}
\end{equation*}
\noindent Then each $\lvert\lvert \cdot \rvert \rvert_n$ is a Minkowski norm on $\R$. 

Furthermore, 
$$\lvert\lvert f \rvert\rvert=\max\{\lvert\lvert f(n)\rvert \rvert_n: n\in \mathbb{N}\} $$
is a Minkowski norm on $V$. Now consider the sequence of functions $(f_n)$ such that $f_n(n)=\frac{1}{n}$ for all $n$, and $f_n(m)=0$ if $m\neq n$. Then $\lvert\lvert f_n \rvert \rvert \to 0$ as $n\to \infty$, but $\lvert \lvert -f_n\rvert\rvert =1$ for all $n$. Thus, the norm $\Vert \cdot \Vert$ is not convergence-symmetric.    
\end{example}

\begin{definition}[Convergence-symmetric Minkowski norm]
A Minkowski normed space is said to be convergence-symmetric if $\lvert\lvert v_n \rvert \rvert\to 0$ implies 
$\lvert\lvert -v_n \rvert \rvert \to 0$ as $n\to \infty$, where $v_n$ is a sequence in $V$.
\end{definition}
If $(V,\lvert\lvert\cdot\rvert\rvert)$ is a Minkowski normed space we define $d_{\lvert\lvert\cdot\rvert\rvert}(x,y)=d(x,y)=\lvert \lvert y-x\rvert\rvert$. Observe that $d_{\lvert\lvert\cdot\rvert\rvert}(x,y)$ is a (not necessarily symmetric) metric on $V$. It is symmetric if and only if $\lvert\lvert\cdot\vert\rvert$ is a norm.

\begin{proposition}
    The following are equivalent for a Minkowski normed space   $(V,\lvert\lvert\cdot\rvert\rvert)$.

\begin{enumerate}
\item 
$(V,\lvert\lvert\cdot\rvert\rvert)$ is convergence-symmetric.
\item  For every sequence \((x_n)\) in \(V\) and every \(x\in V\),
\[
d_{\|\cdot\|}(x_n,x)\to 0
\quad\text{if and only if}\quad
d_{\|\cdot\|}(x,x_n)\to 0.
\]
\item 
For every sequences $(x_n),(y_n)$ in $V$, we have $d_{\lvert\lvert\cdot\rvert\rvert}(x_n, y_n)\to 0$ if and only if $d_{\lvert\lvert\cdot\rvert\rvert}(y_n, x_n)\to 0$ as $n\to \infty$. \item 
There exists $c>0$ such that 
$$\lvert\lvert -x \rvert\rvert \leq c \rvert \rvert x \rvert \rvert $$
\noindent for all $x\in V$.
\end{enumerate}
    \begin{proof}
    	It is easy to see that (4)$\Rightarrow$ (3) $\Rightarrow$ (2) $\Rightarrow$ (1). Let us show that (1) implies (4). Assume that there exist a sequence $(c_n)$ of positive  real numbers and  a sequence $(y_n)\in V$ such that 
    	$$\lim_{n\to \infty} c_n=\infty\ \text{and}\ \lvert\lvert - y_n \rvert\rvert > c_n\lvert \lvert y_n \rvert \rvert $$
    	 for each $n$. Let us take $$x_n=\frac{y_n}{c_n\lvert\lvert y_n \rvert\rvert}.$$
    	Then we have $\lvert\lvert  x_n \rvert \rvert \to 0$, but $\lvert\lvert -x_n\lvert\lvert \geq 1$ for each $n$. 
    	\end{proof}
\end{proposition}

It follows that a Minkowski norm $\lvert\lvert \cdot\rvert\rvert$ is convergence-symmetric
if and only if the corresponding metric $d_{\lvert\lvert\cdot\rvert\rvert}$ is convergence-symmetric if and only if  $d_{\lvert\lvert\cdot\rvert\rvert}$ satisfies Busemann's axiom. 
\begin{definition}[Equivalent Minkowski norms]
Two Minkowski norms $\lvert\lvert \cdot\rvert \rvert $ and $\lvert\lvert \cdot\rvert\rvert '$ on $V$ are said to be equivalent if there exist constants $c>0$ and $C>0$ such that 
$$c\lvert\lvert x\rvert\rvert'\leq \lvert \lvert x\rvert \rvert \leq C \rvert\rvert x \rvert\rvert'$$
for all $x\in V$.
\end{definition}
Thus, if $x\mapsto \lvert\lvert x\rvert\rvert $ is a convergence-symmetric Minkowski norm, then this Minkowski norm and the Minkowski norm $$x\to \lvert\lvert -x\rvert\rvert.$$
are equivalent.

Let \((V,\|\cdot\|)\) be a Minkowski normed space. Consider the following functions
\(V\to \mathbb R\):
\[
\|v\|_a=\frac12(\|v\|+\|-v\|),\qquad
\|v\|_m=\max\{\|v\|,\|-v\|\}.
\]
It is easy to see that \(\|\cdot\|_a\) and \(\|\cdot\|_m\) are equivalent norms on \(V\).

If \(\|\cdot\|\) and \(\|\cdot\|'\) are equivalent Minkowski norms, then
their arithmetic symmetrisations \(\|\cdot\|_a\) and \(\|\cdot\|'_a\) are
equivalent norms. Conversely, for convergence-symmetric Minkowski norms,
equivalence of the arithmetic symmetrisations implies equivalence of the
original Minkowski norms.

The following result is an immediate  consequence of the above definitions. 
\begin{remark}
let  $(V,\lvert\lvert \cdot\rvert\rvert)$ be  a  real vector space with a Minkowski norm. Then  
$$d_{\lvert\lvert\cdot \rvert\rvert_{a}}=d_{\lvert\lvert\cdot \rvert\rvert,\mathrm{arith}},$$
$$d_{\lvert\lvert\cdot \rvert\rvert_{m}}=d_{\lvert\lvert\cdot \rvert\rvert,\mathrm{max}}$$
\end{remark}
We shall denote $d_{ \lvert\lvert\cdot \rvert\rvert,\mathrm{arith}}$ and $d_{\lvert\lvert \cdot\rvert\rvert,max}$ by $d^a$ and $d^m$, respectively. 

\begin{remark}

\label{rm: topology}
The topology on $V$ generated by a convergence-symmetric Minkowski norm $\lvert\lvert \cdot \rvert\rvert$ is, by definition, the topology generated by the convergence-symmetric metric $d_{\lvert\lvert\cdot \rvert\rvert}$, and it coincides 
with the topology generated by the metric $d^a$ (or $d^m$).
 
 Observe that for a Minkowski normed  space $(V,\lvert\lvert \cdot \rvert\rvert)$ the following inequality holds for $v, w\in V$:
$$\lvert \ \lvert\lvert w\rvert\rvert -\lvert\lvert v \rvert\rvert \ \rvert \leq \max(\lvert\lvert w-v\rvert\rvert,\ \lvert\lvert v-w\rvert\rvert)=d^m(v,w) $$
It follows that if $(V,d)$ is convergence-symmetric, then the function $v\to \lvert\lvert v\rvert\rvert$ is continuous. 
\end{remark}

\begin{remark}
    Let $(V,\lvert\lvert\cdot \rvert\rvert)$ be a vector space with a Minkowski norm   and $d$  the corresponding metric. Then,  for all $x,y, a \in V$ and for all $\alpha\geq 0$,

    \begin{enumerate}
        \item 
        $d(x+a,y+a)=d(x,y)$;
        \item 
        $d(\alpha x,\alpha y)=\alpha d(x,y)$.
    \end{enumerate}
\end{remark}

\begin{example}[Minkowski functional]
	We describe a procedure to obtain Minkowski norms on $\R^n$. Let $\Omega\subset \mathbb{R}^n$ be a convex body, that is, a closed bounded convex subset of $\R^n$ which contains the origin in its interior. Then the following function is a Minkowski norm:
	
	\[
	p_\Omega(v)=\inf\left\{t>0:\frac{v}{t}\in\Omega\right\}.
	\]
	
	This Minkowski norm is called the \emph{Minkowski functional} of the set $\Omega$;  see \cite[p. 19]{2012-Hilbert}.
	
Conversely, if \(F\) is a Minkowski norm on \(\mathbb{R}^n\), define
\[
\Omega=F^{-1}([0,1])=\{x\in \mathbb{R}^n \mid F(x)\leq 1\}.
\]
Then \(\Omega\) is a convex body containing the origin in its interior.
Moreover, the Minkowski functional associated with \(\Omega\) satisfies
\[
p_{\Omega}=F.
\]
In this way, one obtains a bijection between convex bodies in
\(\mathbb{R}^n\) containing the origin in their interiors and Minkowski
norms on \(\mathbb{R}^n\). See \cite{2012-Hilbert} for details.

We prove that any Minkowski norm on \(\mathbb{R}^n\) is convergence-symmetric.
Let \(F\) be a Minkowski norm on \(\mathbb{R}^n\), and let \(\Omega\) be the
corresponding convex body. For each non-zero vector \(v \in \mathbb{R}^n\),
let \(b(v)\) be the intersection point of the boundary of \(\Omega\) with the
half-line starting at the origin in the direction of \(v\). Let \(|v|\) denote
the Euclidean norm of \(v\). It is not difficult to see that
\[
F(v)=p_{\Omega}(v)= \frac{|v|}{|b(v)|}.
\]

Since the origin lies in the interior of \(\Omega\), and since \(\Omega\) is
bounded, it follows that \(|v_n| \to 0\) if and only if \(p_{\Omega}(v_n)\to 0\).
It follows immediately that \(F=p_{\Omega}\) is convergence-symmetric. 
Furthermore, since \(\mathbb{R}^n\) is complete with respect to the arithmetic symmetrisation of \(F\), we see that \((\mathbb{R}^n,F)\) is complete.
 
\end{example}

\begin{example}
The set of convex, closed bounded subsets of $\R$ which contain $0$ as an interior point  consists of the  intervals $[a,b]$ where 
$-\infty<a<0<b<\infty$. Thus, for each such interval, the following function is a Minkowski norm on $\mathbb{R}$:
\begin{equation*}
p(x)=
\begin{cases}
	\dfrac{x}{b}, & x\geq 0,\\[4pt]
	\dfrac{x}{a}, & x<0.
\end{cases}
\end{equation*}
\noindent The above discussion show that any Minkowski norm on $\R$ is of this form.
\end{example}

	\section{Finsler structures and convergence-symmetry}\label{s:Finsler}
	After Minkowski norms, it is natural to study Finsler structures. We recall the following.

\begin{definition}[Finsler structure]
	A Finsler structure on a differentiable manifold $M$ is a function $F: TM \to [0,\infty)$ such that
	\begin{enumerate}
		\item 
	$F$ is continuous,
	\item
	for each $x\in M$, $F\lvert_{T_xM}$ is a Minkowski norm.
	\end{enumerate}
\end{definition}

Let $F$ be a Finsler structure on a manifold $M$. For each $C^1$ path $c: [a,b]\to M$, we define the length of $c$ by the formula

$$l(c)=l_F(c)=\int_{a}^{b}F(\dot{c}(t))dt.$$

\begin{definition}[Finsler metric]
	A weak metric $d$ on a differentiable manifold $M$ is called Finsler if it is the length metric associated with a Finsler structure, that is, if there exists a Finsler structure $F$
	on $M$ such that for every $x,y \in M$ we have
	$$d(x,y)= \inf\{l_F(c)\}$$
where $c$ ranges over all piecewise $C^1$ paths such that $c(0)=x$ and $c(1)=y$.
\end{definition}

Note that a Finsler metric on $M$ is generally asymmetric. It is symmetric if the Minkowski norms on $T_xM (x\in M)$, are  norms. In \cite{SOP2} it is proved that the topology induced by a Finsler metric on $M$ coincides with the topology of the manifold $M$. In particular, it follows that a Finsler weak metric is a metric. Moreover, in the same paper, the authors prove that such a metric satisfies 
Busemann's axiom.

\begin{proposition}[Finsler metrics on compact manifolds are convergence-symmetric] \label{Finsler-convergence-symmetric} 
  Let $M$ be a compact $n$-dimensional differentiable manifold, and let $F$ a Finsler structure on $M$. Then the associated distance function on $M$ is convergence-symmetric.
  \end{proposition}
 \begin{proof}
From the description of Minkowski norms on $\R^n$, it follows that for each $x$ in $M$ there exists a real number $c_x>0$ such that $$F(x,-v)\leq c_x F(x,v)$$
for all $v\in T_xM$. The compactness of $M$ and the continuity of $F$ imply that there exists a constant $c>0$ such that $$F(x,-v)\leq c F(x,v),$$
 for all $x \in M$ and for all $v \in T_xM$. 
 Let $d_F$ be the induced weak metric. Then, $$d_F(y,x)\leq c d_F(x,y)$$
for all $x,y \in M$. This implies that $d_F$ is convergence-symmetric. Since \(M\) is compact and the topology of \(M\) coincides with the topology induced by
\(d_F\), the metric space \((M,d_F)\) is compact. Hence it is complete.
\end{proof}

 Furthermore note that $d_F$ satisfies the following property: for any two sequences $(x_n)$ and $(y_n)$ in $M$, the sequence $(d_F(x_n,y_n))$ is bounded if and only if the sequence $(d_F(y_n,x_n))$ is bounded.

In the next section, we continue the study of covergence-symmetric metric spaces.

\section{Arzel\`a-Ascoli, Baire category and Hopf--Rinow Theorems for convergence-symmetric spaces}\label{section:Ascoli}
Let $(X,d)$ be a  metric space. A sequence of functions $f_n: [a,b]\to X$ is said to be uniformly equicontinuous if for every $\epsilon>0$, there exists a $\delta>0$ such that 
$\lvert t -t'\rvert<\delta$ implies $d(f_n(t),f_n(t'))<\epsilon$
for all $n\in \mathbb{N}$.

\begin{theorem}
	\label{thm:arzela}
	Let \((X,d)\) be a compact convergence-symmetric metric space and let
	\((f_n)\) be a uniformly equicontinuous sequence of functions
	\([a,b]\to X\). Then there is a subsequence \((f_{n_k})\) that converges
	uniformly.
\end{theorem}

\begin{proof}
From the assumptions, it follows that $(f_n)$ is  uniformly equicontinuous with respect to the metric $d_{\mathrm{arith}}$. Since $(X,d_{\mathrm{arith}})$ is compact  (Proposition \ref{prop:compactness}),   the usual Arzel\`a--Ascoli Theorem implies that there is a subsequence that converges uniformly with respect to $d_{\mathrm{arith}}$. It follows that this subsequence converges uniformly with respect to $d$. 
\end{proof}

We pass now to the Baire category theorem for convergence-symmetric metric spaces.
If $A$ is a subset of a topological space $X$, then the interior of $A$ is the union of all open subsets of 
$A$. Thus $A$ has empty interior if and only if its complement is dense in $X$.

\begin{definition}
	A topological space \(X\) is called a Baire space if for every sequence
	\((A_n)_{n\geq 1}\) of closed subsets of \(X\) with empty interior, the union
	\[
	\bigcup_{n=1}^{\infty} A_n
	\]
	has empty interior.
\end{definition}

\begin{theorem}
    Let $(X,d)$ be a complete convergence-symmetric metric space. Then $X$ is a Baire space.
    \begin{proof}
        If $(X,d)$ is complete, then $(X,d_{\mathrm{arith}})$ is complete. Therefore $(X,d_{\mathrm{arith}})$ is a 
        Baire space. But the topology $\mathcal{T}_d$ coincides with the topology on $X$ generated 
        by $d_{\mathrm{arith}}$. Thus $X$ is a Baire space.
    \end{proof}
\end{theorem}

First, we state a version of the Hopf--Rinow theorem proved in Busemann's monograph \cite{Busemann1970}. We then establish a version of the Hopf--Rinow theorem adapted to the setting of convergence-symmetric spaces. Note that the result in \cite{Busemann1970} is more general than ours.
We start with the definition of length in a metric space satisfying Busemann's axiom.

\begin{definition}[Paths, lengths of paths and rectifiable paths]
	Let $(X,d)$ be a metric space satisfying Busemann's axiom. A continuous
	map $c: [a,b]\to X$ is called a {\it path}. The {\it length} $l_d(c)=l(c)$
	of a path $c$ is 
	$$l(c)=\sup \{\sum_{i=0}^{n-1}d(c({t_i}),c(t_{i+1}))\},$$
	\noindent where the supremum is taken over all partitions $a=t_0\leq t_1\leq \dots t_{n-1}\leq t_n$ of the interval $[a,b]$. The path $c$ is said to be {\it rectifiable} is it has finite length.
\end{definition}

We now start to describe the basic properties of the length function.

\begin{remark}
	
Let $(X,d)$ be a metric space satisfying Busemann's axiom and let $c: [a,b]\to X$ be a path.

\begin{enumerate}
	\item 
	$l(c)\geq d(c(a),c(b))$ and $l(c)=0$ if and only if $c$ is the constant map;
	\item 
	 if \(\phi:[a',b']\to [a,b]\) is weakly increasing with
	\(\phi(a')=a\) and \(\phi(b')=b\), then
	\[
	l(c)=l(c\circ \phi);
	\]
	\item 
	if $c$ is concatenation of two paths $c_1$ and $c_2$, then $l(c)=l(c_1)+l(c_2)$.

	\end{enumerate}
\end{remark}

\begin{definition}[Length metric space]
	Let $(X,d)$ be a metric space satisfying Busemann's axiom. Then $d$ is said to be a length metric if for any $x,y$ in $X$  $d(x,y)$ is
	equal to the infimum of the length of rectifiable paths joining $x$ to $y$.
\end{definition}

\begin{example}
	Let $M$ be a differentiable manifold, and  $F$  a Finsler structure on $M$. Let us denote the metric on $M$ induced by $F$ by $d_F$. This means that $d_F(x,y)$
	is the infimum of the lengths of the paths joining $x$ to $y$. As we noted earlier, the authors proved in \cite{SOP2} that $d_F$  is a metric which satisfies Busemann's axiom. Let us denote the length  of a path $c: [a,b]$ computed with respect to the metric $d_F$ by $l_{d_F}$. A result of Busemann--Mayer implies that $l_{d_F}(c)=l_F(c)$ for any piecewise $C^1$-path $c$ (see \cite{BM}). It follows that the metric $d_F$ is a length metric.
\end{example}

\begin{definition}
	Let $(X,d)$ be a metric space. A function $c: I\to X$ is called a (constant speed) geodesic if there is a constant $\beta>0$ such that 
	$$d(c(t),c(t'))=\beta (t' -t)$$
	\noindent for all $t,t' \in [a,b]$ with $t'\geq t$, where $I$ in an interval. In this case we say that $c$ is a (constant speed) geodesic from $c(a)$ to $c(b)$ provided that $I=[a,b]$. The space $(X,d)$ is called a geodesic (metric) space if for any $x$ and $y$ in $X$, there is a geodesic joining $x$ to $y$.
\end{definition}

\begin{remark}
	Assume that $(X,d)$ is a metric space satisfying Busemann's axiom. Then a geodesic $c: [a,b]\to X$ is a continuous map. This can be proved by using the sequential characterisation of continuity (see Proposition \ref{prop:seq-cont}).
\end{remark}

Let $(X,d)$ be a metric space satisfying Busemann's axiom. Let $p\in X$ and $\rho>0$. Let us use the following notation:
$$\overline{S^+}(p,\rho)=\{x\in X: d(p,x)\leq \rho\}.$$

\begin{theorem}[{\cite[p.~4]{Busemann1970}}]
	Let $(X,d)$ be locally compact length space satisfying Busemann's axiom. Then the following are equivalent:
	\begin{enumerate}
		\item 
		$\overline{S^+}(p,\rho)$ is compact 
		for each $x\in X$ and $\rho>0$.
		\item 
		The space $(X,d)$ is forward complete.
		\item 
		Each geodesic $[a,b)\to X$ can be extended to a geodesic $[a,b]\to X$.
	\end{enumerate}
	Furthermore, if these conditions are satisfied, then $X$ is a geodesic space.
	\end{theorem}

Note that a generalisation of the above theorem for a larger class of metric spaces was obtained in \cite{Mennucci2014}. We return to the proof of the Hopf--Rinow Theorem in the setting of convergence-symmetric metric spaces.
	
\begin{proposition}\label{prop:rectifiable}
Let $(X,d)$ be a metric space satisfying Busemann's axiom and let $c: [a,b]\to X$ be a continuous map.  If $c$ is rectifiable of length $l$, then the function $\lambda: [a,b]\to [0,l]$ defined  by $$\lambda(t)=l(c\lvert_{[a,t]})$$
\noindent is a continuous weakly increasing function.
\begin{proof}
It is enough to prove  that for any $\epsilon>0$, there is a partition of the interval $[a,b]$ such that  length of $c$ restricted each of these subintervals is less than or equal to $\epsilon$. Observe that  since $X$ satisfies 
Busemann's axiom,  $c:[a,b]\to X$ is uniformly continuous (see Theorem \ref{thm:uniform-continuity}).  Thus there exists a $\delta >0$ such that $d(c(t),c'(t))<\epsilon/2$ whenever $\lvert t-t'\rvert <\delta$. Since $l(c)$ is finite we can find a partition $a=t_o\leq t_1\dots \leq t_k=b$
such that 
$$\sum_{i=0}^{k-1}d(c(t_i),c(t_{i+1}))> l(c)-\epsilon/2.$$
\noindent By taking a refinement of this partition, we may assume that $\lvert t_i-t_{i+1}\rvert <\delta$, and this mplies that $d(c(t_i),c(t_{i+1}))<\epsilon/2$. Then we have
$$l(c)=\sum_{i=0}^{k-1}l(c\lvert_{[t_i,t_{i+1}]})\geq \sum_{i=0}^{k-1}d(c(t_i),c(t_{i+1}))>l(c)-\epsilon/2.
$$
\noindent Since $l(c\lvert_{[t_i,t_{i+1}]})\geq d(c_{t_i},c_{t_{i+1}})$, we see that 
$$l(c(t_i),c(t_{i+1})))-d(c(t_i),c({t_{i+1}}))\leq \epsilon/2,$$
\noindent and this implies $l(c(t_i),c(t_{i+1}))<\epsilon$.

\end{proof}
\end{proposition}

\begin{definition}
	A path $c:[a,b]\to X$ is parametrised proportionally to arc-length if 
	\[
	\lambda(t)=\frac{t-a}{b-a}\,l(c)=\frac{t-a}{b-a}\,\lambda(b).
	\]
	\noindent for all $t\in [a,b]$, where $\lambda$ is the function defined in Proposition \ref{prop:rectifiable}.
\end{definition}
\begin{remark}
	Let $c: [a,b]\to X$ be a rectifiable path. It is easy to see that there is a reparametrisation 
	$c':[0,1]\to X$ of $c$ which is proportional to arc-length.
\end{remark}

Let $[a,b]$ be a compact interval and $(X,d)$ a convergence-symmetric metric space. If $c: [a,b]\to X$ is a continuous path then its image is a compact set.
Let $\mathcal{C}([a,b],X)$ be the set of paths in $X$ with domain $[a,b]$. The function

$$d_{\mathcal{C}}: \mathcal{C}([a,b],X)\to \R$$
  given by $$d_{\mathcal{C}}(c,c')=\sup_{t\in [a,b]} d(c(t),c'(t))$$
  is a convergence symmetric-metric.

\begin{definition}[Lower semi-continuous function]
	Let $E$ be a topological space and $\bar{\R}=\R\cup \{-\infty,\infty\}$. A function $E\to \bar{\R}$ is called lower semi-continuous at $x_0 \in E$ if for any real number $m< f(x_0)$ there is a neighbourhood $W$ of $x_0$ in $E$ such that for all $x \in W$, we have $m \leq f(x)$. A function is called lower semi-continuous if it is lower semi-continuous at each point in its domain.
\end{definition}

The following proposition and lemma are taken from pages 30--31 of \cite{Papadopoulos2005}.

\begin{proposition}
	\label{prop:upper-limit}
	Let $(f_i)_{i\in I}$ be a family of maps from $E$ to $\bar{\R}$ and let $f: E\to \bar{\R}$ be the upper limit of this family, that is, $f(x)=\sup_{i\in I}f_i(x)$ for all $x\in E$. Let $x_0\in E$. If each $f_i$ is lower semi-continuous at $x_0$, then $f$ is lower semi-continuous at $x_0$.
	\end{proposition}

\begin{lemma}
	\label{lemma:athanase}
Let $E$ be a  topological space and $f: E\to \bar{\R}
$ be  a function which is lower semi-continuous at $x \in E$. Then for any sequence $(x_n)$ that converges to $x$, we have 
$$f(x)\leq \lim\inf_{n\to \infty}f(x_n).$$
\end{lemma}

\begin{theorem}
	\label{theorem.lower-semicont}
	 If the space $(X,d)$ is convergence-symmetric, then the length function $L:\mathcal{C}([a,b],X)\to \bar{\R} $ defined by $L(c)=l(c)$ is lower semi-continuous.
	
	\begin{proof}
		For each $t\in [a,b]$, the map $\mathcal{C}([a,b],X)\mapsto X$ given by $t\mapsto c(t)$ is continuous. This follows from the inequality $d(c(t),c'(t))\leq d_{\mathcal{C}}(c,c')$. It follows from Corollary \ref{cor:metric-cont} that for any $t $ and $t'$ in $[a,b]$, the function $c\to d(c(t),c(t'))$ is continuous. Now the definition of the length of $c$ and Proposition \ref{prop:upper-limit} imply that $L$ is lower semi-contiuous.
	\end{proof}
\end{theorem}

\begin{corollary}
	\label{cor:lower-semicont-1}
	
Let $(c_n: [a,b]\to X)$ be a sequence of paths converging uniformly to a function  $c: [a,b] \to X$. If $(X,d)$ is convergence-symmetric then $c$ is continuous and 
$$l(c)\leq \lim\inf_{n\to \infty}l(c_n).$$

\begin{proof}
	Continuity of $c$ follows from Proposition \ref{prop:uniform-convergence}. The rest of the statement follows from Lemma \ref{lemma:athanase} and Theorem 
	\ref{theorem.lower-semicont}.
\end{proof}
\end{corollary}

 A metric space $(X,d)$ satisfying Busemann's axiom is said to be locally compact when $X$ is locally compact with respect to the topology 
generated by $d_{\mathrm{arith}}$. We denoted this topology by $\mathcal{T}_d$. See Definition \ref{topology} and Lemma \ref{Busemann same top}.

\begin{theorem}
	\label{thm:Hopf--Rinow-0}
	Let $(X,d)$ be a convergence-symmetric length space. 
	If $X$ is  complete and locally compact, then every bounded closed subset of $X$ is compact.
	\end{theorem}
\begin{proof}
We give a proof imitating the proof in the case of  symmetric length space, which can be found in  \cite[Théorème 1.10]{GLP} and also \cite[Chap I, Theorem 3.7]{BH}. 

It suffices to prove the statement for a closed (right) ball $\overline{S^+}(p, \rho)$ for every $p \in X$ and every $\rho >0$.
For simplicity, we denote the closed ball $\overline{S^+}(p, \rho)$ by $\bar S(\rho)$.
Let $\rho_0=\sup \{\rho : \bar S(\rho) \text{ is compact}\}$.
Since $(X,d)$ is locally compact, $\rho_0$ is positive.
We are done if we can show that $\rho_0=\infty$.

Suppose that $\rho_0$ is finite.
We consider the ball $\bar S(\rho_0)$.
Suppose first that $\bar S(\rho_0)$ is compact.
By local compactness, for each point $x$ in $\bar S_{\rho_0}$, there exists $\epsilon_x >0$ such that $\overline{S^+}(x, \epsilon_x)$ is compact.
The open sets $\{S^+(x, \epsilon_x/2)\}$ cover $\bar S(\rho_0)$, hence we can find a finite sub-covering $\{S^+(x_1, \epsilon_{x_1}/2), \dots , S^+(x_k, \epsilon_{x_k}/2)\}$ of $\bar S(\rho_0)$.
Then $B:=\bar S_{\rho_0} \cup \overline{S^+}(x_1, \epsilon_{x_1}) \cup \dots \cup \overline{S^+}(x_k, \epsilon_{x_k})$ is compact.

Let $\epsilon_0=\min\{\epsilon_1, \dots , \epsilon_k\}$.
We consider the ball $\bar S(\rho_0+\epsilon_0/3)$.
Since $(X,d)$ is a length space, for any point $x \in \bar S(\rho_0+\epsilon_0/3)$, there exists an arc $\alpha$ connecting  $p$ to $x$ with length less than $\rho_0+\epsilon_0/2$.
We can take a point $y$ on $\alpha$ such that $d(y, x)=\epsilon_0/2$.
Then, since $d(p, y) \leq \rho_0$, we see that $y$ is contained in $\bar S(\rho_0)$.
It follows that $y$ is contained in $S^+(x_i, \epsilon_{x_i}/2)$ for some $i=1, \dots , k$, and hence that $x$ is contained in $\overline{S^+}(x_i, \epsilon_{x_i}) \subset B$.
Thus we have shown that $\bar S(\rho_0+\epsilon_0/3)$, which is a closed subset of $B$, is also compact, contradicting our definition of $\rho_0$.
Therefore, if $\rho_0$ is finite, then $\bar S(\rho_0)$ cannot be compact.

We shall prove that $\bar S(\rho_0)$ is sequentially compact, which implies that $\bar S(\rho_0)$ is compact, for the topology of $X$ is induced by the symmetric metric $d_{\mathrm{arith}}$.
This will contradict what we have just proved above, and will imply that $\rho_0=\infty$.
Let $(x_n)$ be a sequence in $\bar S(\rho_0)$.
If $\liminf d(p,x_n) < \rho_0$, then there exists a subsequence of $(x_n)$ contained in $\bar S(\rho_1)$ with $\rho_1 < \rho_0$, which is assumed to be compact, and we are done.
Therefore, we can assume that $\lim d(p,x_n)=\rho_0$ and passing to a subsequence, that $d(p, x_n) >\rho_0-1/n$.

Since $(X,d)$ is a length space, for any $m \in \mathbb N$, there is an arc $\alpha^m_n$ connecting $p$ to $x_n$ with length less than $\rho_0+1/2m$.
%Since $d(p,x_n)> \rho_0-1/n$, the length of $\alpha^m_n$ is greater than $\rho_0-1/n$.
We choose  a point $y_n^m$ on $\alpha_n$ such that $1/m<d(y_n^m, x_n)<2/m$.
Then $d(p,y_n^m) < \rho_0-1/2m$, for otherwise $d(p, x_n) > \rho_n+2/m$.
Therefore, $(y^m_n)_n$ is contained in $\bar S(\rho_0-1/2m)$, which is compact. Passing to a subsequence, $(y^m_n)_n$ converges.
By a diagonal argument, we can take a subsequence of $(x_n)$ so that $(y^m_n)_n$ converges for every $m$.
Abusing notation, we denote such a subsequence of $(x_n)$ again by $(x_n)$.

We shall next show that for any $\epsilon >0$, there exists an integer $M$ such that $d(x_n, y^m_n)< \epsilon$ for every $n$ and $m >M$.
Suppose not.
Then, passing to a subsequence, there exists an integer $m(n)$ converging to $\infty$ as $n \to \infty$ such that $d(x_n, y^{m(n)}_n) \geq \epsilon$ for all $n$.
On the other hand, by our choice of $y^{m(n)}_n$, we have $d(y^{m(n)}_n, x_n) < 2/m(n) \to 0$.
This contradicts the assumption that $(X,d)$ is convergence-symmetric.
%Passing to a subsequence of $(y^m_n)_m$, and by a diagonal argument again, we can assume that $d(x_n, y_n^m) <1/m$ and $d(y_n^m, x_n)<1/m$ for every $n$.

Now, let us show that $(x_n)$ is a Cauchy sequence.
Let $\epsilon>0$ be any given positive number.
We can take a sufficiently large $m$ so that $d(x_n, y^m_n) < \epsilon$ and $d(y^m_{n+l}, x_n) < \epsilon$  for every $n$, as shown above.
Fix such $m$.
Since $(y^m_n)_n$ converges, it is a  Cauchy sequence.
Therefore, for the given $\epsilon >0$, there exists an integer $N$  such that $d(y_n^m, y_{n+l}^m)< \epsilon$ for any $n > N$ and $l\geq 0$.
Then, we have $d(x_n, x_{n+l})\leq d(x_n, y^m_n)+d(y^m_n, y^m_{n+l})+d(y^m_{n+l}, x_{n+l}) < 3\epsilon$.
This means that $(x_n)$ is a  Cauchy sequence. B completeness of $(X,d)$, it converges.
Thus we have shown that $\bar S_{\rho_0}$ is sequentially compact, which is a contradiction.
This shows that $\rho_0$ must be $\infty$.	
%	Let $A$ be a bounded closed subset of $X$. Then $A$ is bounded and closed with respect to $d_{\mathrm{arith}}$. Since $(X,d_{\mathrm{arith}})$ is complete and locally 
%	compact (see Theorem \ref{thm-crucial}), it follows from the Hopf--Rinow Theorem that $A$ is compact with respect to the metric $d_{\mathrm{arith}}$. Proposition \ref{prop:compactness} implies that $A$ is compact with respect to the metric $d$.
\end{proof}

\begin{theorem}
	\label{thm:Hopf--Rinow-1}
	Let $(X,d)$ be  a complete, locally compact, convergence-symmetric length metric space. If for all sequences $(x_n)$, $(y_n)$ in $X$, 
{the sequence} $(d(x_n,y_n))$ {is bounded implies that the sequence } $(d(y_n,x_n))$ is bounded  as well, then $X$ is a geodesic space.
\begin{proof}
	Let $a$ and $b$ be two distinct points of $X$. For every integer $n\geq 1$, there is a  path $c_n: [0,1]\to X$ which joins $a$ to $b$, which is  parametrised proportional to arc-length and such that 
	$$l(c_n)\leq d(a,b)+1/n.$$
	
	We show that the family $(c_n)$ is equicontinuous. Assume for contradiction that the sequence $(c_n)$ is not equicontinuous. Then there exist sequences $(t_k)$, $(t'_k)$ in $[0,1]$ and a subsequnce $(c_{n_k})$ such that $\lvert t_k-t'_k\lvert\to 0$ and $d(c_{n_k}(t_k),c_{n_k}(t'_k))$ does not converge to $0$ as $k\to \infty$.
	Assume that $t'_{k_l}\geq t_{k_l}$ for   infinitely many $k_l$, where $k_{l+1}>k_l$. Then we have
	$$\lvert t'_{k_l}-t_{k_l} \rvert=\frac{l(c_{n_{k_l}}\lvert_{[t_{k_l},t'_{k_l}]})}{l(c_{n_{k_l}})}\geq \frac{d(c_{n_{k_l}}(t_{k_l}),c_{n_{k_l}}(t'_{k_l}))}{d(a,b)+1}.$$
	
	\noindent Taking the limit in the above inequality, we get a contradiction. It follows that we may assume that $t_k > t'_k$ for all but finitely many $k$. In this case, exchanging the roles of $t_k$ and $t'_k$ in the above inequality, we have 
	$$\lim_{k\to \infty}d(c_{n_k}(t'_k),c_{n_k}(t_k))=0.$$
	\noindent By convergence-symmetry, it follows that 
		$$\lim_{k\to \infty}d(c_{n_k}(t_k),c_{n_k}(t'_k))=0$$
 as well, contradicting the assumption. Thus $(c_n)$ forms an equicontinuous family of paths. 

We now show that the set 
$$C=\{c_n(t): t\in [0,1], n\geq 1\}$$
is bounded. Assume for contradiction that $C$ is not bounded. Then there exist sequences $(t_k)$ and $(t'_k)$ in $[0,1]$ and a subsequence $(c_{n_k})$ such that 
$$\lim_{k\to \infty}d(c_{n_k}(t_k),c_{n_k}(t'_k))=\infty.$$
Assume further that $t'_{k_l}\geq t_{k_l}$ for  infinitely many $k_l$, where $k_{l+1}>k_l$. 
Then we have
$$d(c_{n_{k_l}}(t_{k_l}),c_{n_{k_l}}(t'_{k_l})\leq l(c_{n_{k_l}})\leq d(a,b)+1.$$
 Taking the limit of the above inequality, we get a contradiction to the assumption. It follows that we may assume that $t_k>t'_k$ for all but finitely many  $k$. In this case, exchanging the roles of $t_k$ and $t'_k$ in the above inequality, we see that the sequence 
$$(d(c_{n_k}(t'_k),c_{n_k}(t_k)))$$
\noindent is bounded. It follows that the sequence 
$$(d(c_{n_k}(t_k),c_{n_k}(t'_k)))$$
is bounded as well, contradicting our assumption. Thus $C$ is a bounded set.

Since $C$ is bounded, its closure is bounded as well. By Theorem \ref{thm:Hopf--Rinow-0}, we see that the closure of $C$ is compact. Theorem \ref{thm:arzela} implies that there is  a subsequence $(c_{n_k})$ that converges uniformly to a function $c:[0,1]\to X$.   Proposition  \ref{prop:uniform-convergence} implies that $c$ is continuous as well. Finally,
Corollary \ref{cor:lower-semicont-1} implies that 
$$l(c)\leq \liminf l(c_{n_k})= d(a,b).$$
But $l(c)\geq d(a,b)$, hence $l(c)=d(a,b)$. Without loss of generality, we may assume that $c$ is parametrised by arc-length. 
Now we have, for all $a\leq t\leq t'\leq b$
\begin{align*}
	l(c)&=d(c(a),c(b))\\
	&\leq d(c(a),c(t))+d(c(t),c(t'))+d(c(t'),c(b))\\
	&\leq l(c\lvert_{[a,t]})+l(c\lvert_{[t,t']})+l(c\lvert_{[t',b]})\\
	&= l(c).
	\end{align*}
\noindent 	Therefore we have $d(c(t),c(t'))=l(c\lvert_{[t,t']})$. Since $c$ is parametrised by arc-length, we see that $l(c\lvert_{[t,t']})=t'-t$. It follows that $d(c(t),c(t'))= t'-t$. This completes the proof.
\end{proof}
\end{theorem}

\begin{corollary}
	Let $M$ be a compact manifold and $F$ a Finsler structure on $M$. Consider the metric $d_F$ on $M$ induced by $F$. Let $x$ and $y$ be two points on $M$. Then there exists a geodesic on $M$ joining $x$ to $y$.
	\begin{proof}
We know that $(M,d_F)$	is a length metric space. In Section \ref{Finsler-convergence-symmetric} we showed that $(M,d_F)$ satisfies the other 
conditions of Theorem \ref{thm:Hopf--Rinow-1}. It follows that $(M,d_F)$ is a geodesic space.
	\end{proof}
\end{corollary}

\section{Some Thurston-type Finsler metrics and some questions}\label{s:Thurston}

A beautiful example of a Finsler non-compact manifold is Thurston's metric. In this last section, we shall recall this metric for the convenience of the reader and we shall then ask several questions concerning it topology, all related to the topics we addressed in the preceding sections.
\begin{example}[Thurston metric]
    	In his 1985 preprint {\it Minimal Stretch Maps between hyperbolic surfaces} \cite{thurston}, Thurston introduced a metric on Teichm\"{u}ller space, defined as follows:
	
	Let $M$ be a closed orientable surface  of genus $\geq 2$, and 
	$g$ and $h$ two marked hyperbolic surfaces on $M$.  For a homeomorphism $\phi: (S,g)\to (S,h)$, we set
	$$L(\phi)=\sup_{x\neq y}\{d_h(\phi(x),\phi(y))/d_g(x,y)\}.$$

	  Given the two marked hyperbolic structures $g$ and $h$ on $M$ representing elements of the Teichm\"{u}ller space of $M$, we set 
	\begin{equation}\label{eq:infimum}
	L(g,h)=\inf_{\phi\simeq \mathrm{Id}}\log L(\phi)
	\end{equation} 
	\noindent where the infimum is taken over all the  homeomorphisms $\phi$ of $S$ which are homotopic to the identity.
	Thurston proved that
\begin{enumerate}
 \item 
		 $L$ is an asymmetric metric on Teichm\"uller space; 
		\item
		$L$ is a Finsler metric; 
		\item
		there always exist a (non-unique) best Lipschitz homeomorphism between two hyperbolic structures on $M$ which realises the infimum in \ref{eq:infimum};
		\item
		$L(g,h)$ is also given by the quantity $\log \sup_{\gamma\in \mathcal{S}}\frac{l_h(\gamma)}{l_g(\gamma)}$  
		where $\mathcal{S}$ is the set of homotopy classes of non-trivial simple closed curves on $M$ and $l_g(\gamma)$ denotes the length of the geodesic in the class  $\gamma$ with respect to $g$;
		\item
		any two points of Teichmüller space of $M$ can be joined by a geodesic for the metric $L$.
	\end{enumerate}
\end{example}

After Thurston's paper appeared, this metric has been generalised to several settings, including Teichm\"uller spaces of flat tori \cite{2005c}, spaces of hyperbolic surfaces with boundary  (with a modified version of this metric, called the arc metric) \cite{PS2015, HP2021}, 
spaces of Euclidean triangles and spaces of triangulated Euclidean surfaces  \cite{Saglam1, OMP2, SOP1, HLLZ},   
spaces of quadrangulated surfaces equipped with singular flat metrics \cite{Saglam2}, 
spaces of semi-translation surfaces \cite{Wolenski, Shi}, and other settings. The interested reader may refer to the paper 
\cite{PS2015} on the Finsler structure of the Thurston metric, and to
the  survey \cite{Course}. Two other surveys on Thurston's metric appeared recently, see \cite{Pan-Su} and \cite{Xu}. A new survey will appear in \cite{Tokyo-2025}. See also the problem set \cite{S}, which dates from  2015, some of which are still open. Some spaces of surfaces equipped with metrics of  Thurston type are included in the examples below.

\begin{example}[Thurston-type metric in spaces of L-shaped polygons]
  %  \begin{figure}
  %  \includegraphics{Lshape.pdf}
  %  \end{figure}
 An L-shaped polygon is shown in the Figure \ref{rectangles}. More generally, an L-shaped polygon is made up of a union of Euclidean rectangles. The dots, which are the vertices of the rectangles, are regarded as marked points on the boundary of the L-shaped polygon, and we call their union the labels of the polygon.  
 An L-shaped polygon is an example of a Euclidean quadrangulated surface, a class of surfaces studied in \cite{Saglam2}. The space of  L-shaped polygons is parametrised by the four real numbers $a_1,a_2,a_3,a_4$ representing lengths and shown in Figure \ref{rectangles}. The area of this L-shaped polygon is equal to the quantity
    $$a_1a_4+a_3a_4+a_2a_3.$$

\begin{figure}[h]
\centering
  \begin{tikzpicture}[scale=2.5]

  % Düğümler
  \foreach \x/\y in {0/0, 2/0, 0/1, 2/1, 4/0, 4/1, 0/2, 2/2}
    \filldraw[blue] (\x,\y) circle (2pt);

  % Dikdörtgenler
  \draw[thick] (0,0) -- (2,0) -- (2,1) -- (0,1) -- cycle; % Sol alt dikdörtgen
  \draw[thick] (2,0) -- (4,0) -- (4,1) -- (2,1) -- cycle; % Sağ dikdörtgen
  \draw[thick] (0,1) -- (2,1) -- (2,2) -- (0,2) -- cycle; % Üst dikdörtgen

  % Kenar etiketleri
  % Sol alt dikdörtgen
  \node[left] at (0,0.5) {\(a_3\)};
  %\node[right] at (2,0.5) {\(a_1\)};
  \node[below] at (1,0) {\(a_4\)};
 % \node[above] at (1,1) {\(a_4\)};

  % Sağ dikdörtgen
  \node[right] at (4,0.5) {\(a_3\)};
  \node[left] at (2,0.5) {\(a_3\)};
  \node[below] at (3,0) {\(a_2\)};
  \node[above] at (3,1) {\(a_2\)};

  % Üst dikdörtgen
  \node[left] at (0,1.5) {\(a_1\)};
  \node[right] at (2,1.5) {\(a_1\)};
  \node[below] at (1,1) {\(a_4\)};
  \node[above] at (1,2) {\(a_4\)};

\end{tikzpicture}
\caption{L-shaped polygon}
\label{rectangles}
\end{figure}
    
 Assume that $P$ and $P'$ are two L-shaped polygons with labelled vertices and let $f:P \to P'$ be a label-preserving homeomorphism. We define  
 $$L(f)=\sup_{x\neq y \in P}\frac{d_{P'}(f(x),f(y))}{d_{P}(x,y)}.$$
and
$$L(P,P')= \inf_{f \equiv id}\log L(f).$$

In the first formula, $d_P$ is the metric on $P$ (respectively $P'$), defined such that the distance between two points is equal to the length of the shortest piecewise-linear path between them, the length of a path being measured in each rectangle in the usual way.
In the second formula, the infimum is taken over the set of all label-preserving homeomorphisms between the two L-shaped polygons.

Note that the value 
$L(P,P')$ depends only on the isometry  classes  $[P]$ and $[P']$,  of $P$ and $P'$. Thus, $L([P],[P'])$ is well defined. Let us denote the space of isometry classes of  unit area 
L-shaped polygons by $\mathcal{R}_1$.  It was shown in \cite{Saglam2} that $L$ is a metric on $\mathcal{R}_1$.   The following function was also defined in the same paper:
		$$K(P,P')=\log \max_i\{\frac{a_i'}{a_i}\}.$$
\end{example}
 In the next proposition, we summarise some results of \cite{Saglam2} restricted to the case of L-shaped polygons.
\begin{proposition}
	\begin{enumerate}
			\item 
			We have  $L=K$.
			\item
			 $L$ is an asymmetric metric  when it is restricted to the space $\mathcal{R}_1$ of unit area L-shaped polygons.
		
		\item
		A  sufficient condition for a path in $\mathcal{R}_1$ to be a geodesic is given in \cite{Saglam2}, where it is also proved  that any two points on $\mathcal{R}(Q)_1$ are joined by a geodesic.
			
			\item
			The restriction of the metric $L$ to $\mathcal{R}(Q)_1$ is Finsler,  and a formula is given for the associated Minkowski norm of the tangent space of each point in $\mathcal{R}_1$.	
		\end{enumerate}

\end{proposition}

We now check the following.
\begin{proposition}
    The metric $L$ does not satisfy the convergence-symmetry property. 
    \end{proposition}
    \begin{proof}
    To prove the proposition, we consider the two sequences of L-shaped polygons defined by
    \[ P_n := \left( n+1, \frac{1}{n}, n-1, \frac{1}{2n^2} \right), \quad P'_n := \left( n+\sqrt{n}, \frac{1}{n}, n-\sqrt{n}, \frac{1}{2n^{3/2}} \right). \]

   We first check that these L-shaped polygons have all unit area:
\begin{align*}
    (n+1) \cdot \frac{1}{2n^2} + (n-1) \cdot \frac{1}{2n^2} + \frac{1}{n} \cdot (n-1) &= 1, \\
    (n+\sqrt{n}) \cdot \frac{1}{2n^{3/2}} + (n-\sqrt{n}) \cdot \frac{1}{2n^{3/2}} + \frac{1}{n} \cdot (n-\sqrt{n}) &= 1.
\end{align*}
    
    We now examine the behaviour of $K(P_n, P'_n)$ and $K(P'_n, P_n)$ as $n \to \infty$. Computing the ratios of the corresponding parameters, we obtain
    \[
    K(P_n, P'_n) = \log \max \left\{ \frac{n+\sqrt{n}}{n+1}, 1, \frac{n-\sqrt{n}}{n-1}, \frac{\frac{1}{2n^{3/2}}}{\frac{1}{2n^2}} \right\}.
    \]
     Since $\frac{n+\sqrt{n}}{n+1} \to 1$, $\frac{n-\sqrt{n}}{n-1} \to 1$ and $\displaystyle \frac{\frac{1}{2n^{3/2}}}{\frac{1}{2n^2}} \to \infty $ as $n \to \infty$, we have
    \[ K(P_n, P'_n) \to \infty. \]
    
    On the other hand, computing $K(P'_n, P_n)$, we have
    \[
    K(P'_n, P_n) = \log \max \left\{ \frac{n+1}{n+\sqrt{n}}, 1, \frac{n-1}{n-\sqrt{n}}, \frac{\frac{1}{2n^2}}{\frac{1}{2n^{3/2}}} \right\}.
    \]
    The last term is simplified to $\frac{n^{3/2}}{n^2} = \frac{1}{\sqrt{n}}$, which converges to zero as $n \to \infty$, yielding
    \[ K(P'_n, P_n) \to 0. \]
    
    This proves that there exist sequences $(P_n), (P'_n)$ such that $K(P'_n, P_n) \to 0$ when $n\to\infty$ while $K(P_n, P'_n) \to \infty$  when $n\to\infty$, contradicting the convergence-symmetry property. 
\end{proof}

\begin{example}[Spaces of Euclidean triangles] \label{ex:triangles}
    The set of isometry classes of labelled  Euclidean triangles is parametrised by the following subset of $\R^3$:
$$\{(a_1,a_2,a_3)= a_1,a_2,a_3>0, \ a_2+a_3-a_1>0, a_1+a_3-a_2>0, a_1+a_2-a_3>0\}.$$
Here, a positive triple $(a_1, a_2, a_3)$ represents the lengths of the  edges of a triangle.

This set can be identified with $(\R^*_+)^3=\{(A_1,A_2,A_3): A_1, A_2, A_3 >0\}$ via the maps
$$A_1=\frac{a_2+a_3-a_1}{2}$$
$$A_2=\frac{a_1+a_3-a_2}{2}$$
$$A_3=\frac{a_1+a_2-a_3}{2}.$$
The area of a triangle $(a_1,a_2,a_3)$ in terms of $(A_1,A_2,A_3)$ is given by Heron's formula:
$$\mathrm{Area}(a_1,a_2,a_3)=\mathrm{Ar}(A_1,A_2,A_3)=\sqrt{(A_1+A_2+A_3)A_1A_2A_3}.$$

We define the function $$\eta: (\R^*_+)^3\times (\R^*_+)^3\to \R$$ by the formula
$$\eta((A_1,A_2,A_3),(A_1',A_2',A_3'))=\log \max \{A_1'/A_1 ,A'_2/A_2, A_3'/A_3\}.$$

Consider the  subspace $\frak{T}_1\subset (\R^*_+)^3$ consisting  of unit area triangles:
$$\frak{T}_1=\{(A,B,C): \mathrm{Ar}(A,B,C)=1\}.$$

\begin{proposition}
The function $\eta$ is an asymmetric metric on $\frak{T}_1$. 

\end{proposition}

\begin{proof}
The triangle inequality is obvious.We prove nonnegativity.  Assume that $(A,B,C)$ and $(A',B',C')$ are in $\frak{T}_1$ and that $ \eta((A,B,C),(A',B',C'))\leq0$. This means that $A'\leq A$, $B'\leq B$ and $C'\leq C $. Assume, without loss of generality, that $A'<A$. Then, from Heron's formula that we recalled, we have 
$$\mathrm{Ar}(A',B',C')<\mathrm{Ar}(A,B,C),$$
which is a contradiction. It follows that $A=A'$. Thus  $ \eta((A,B,C),(A',B',C'))\leq0$ implies $A=A', B=B'$ and $C=C'$. Asymmetry may be checked by computing the distances between the triples
$(1,1,1)$
 and
$(\sqrt{3}/2, \sqrt{3}/2, 1-\sqrt{3}/2).$
\end{proof}

The asymmetric metric $\eta$  on  $\frak{T}_1$
 was introduced and studied  in \cite{SOP1}. 
We shall show that the space 
$\frak{T}_1$ satisfies the convergence-symmetry property. 
We start with the following lemma:
\begin{lemma}
\label{simple-inequality}
Let $a_1, a'_1,\dots a_n, a'_n$ be positive real numbers. Then
$$\frac
{a'_1+\dots +a'_n}{a_1+\dots+a_n}\leq \max_i\{\frac{a'_i}{a_i}\}.$$
\end{lemma}
\begin{proof}
Indeed, we have
$$a'_1+a'_2+\dots +a'_n=a_1\frac{a'_1}{a_1}+\dots+a_n \frac{a'_n}{a_n}\leq \max_i\{\frac{a'_i}{a_i}\}(a_1+\dots +a_n),$$
\noindent and the result follows.
\end{proof}

\begin{proposition}
The space 
$\frak{T}_1$ satisfies the convergence-symmetry property. 
\end{proposition}
\begin{proof}
We shall prove that for any two sequences
$A(n)=(A_1(n,A_2(n),A_3(n)))$ and $A'(n)=(A'_1(n),A'_2(n),A'_3(n))$  in  $\frak{T}_1$, we have
		$$\eta(A(n),A'(n))\to 0\ \text{ if and only if}\ \log A'_{i}(n)-\log A_{i}(n)\to 0\ \text{ for all} \ i=1,2, 3.$$
	
	If, for each $i$, $\log A'_i(n)-\log A_i(n)\to 0$, then $A_i'(n)/A_i(n)\to 1$ for each $i$. It follows that $\eta(A(n),A'(n))\to 0$.
For the other implication, assume that $\eta(A(n),A'(n))\to 0$. It follows that $\max\{A'_{i}(n)/A_i(n)\}\to 1$.
To prove the converse, we assume that there exists $i'$ such that $\lim_{n\to \infty}\frac{A'_{i'}(n)}{A_{i'}(n)}\neq 1$. We may suppose the following 

\begin{enumerate}
\item
$\lim_{n\to \infty}\frac{A'_1(n)}{A_1(n)}=1$.
\item
$\lim_{n\to\infty}\frac{A'_2(n)}{A_2(n)}$ exists and strictly less than $1$.
\item
$\lim_{n\to \infty}\frac{A'_3(n)}{A_3(n)}$ exists and less than or equal to $1$.
\end{enumerate}
 
 From Lemma \ref{simple-inequality}, we have
\begin{equation*}
\begin{split}
1=&\frac{[\mathrm{Ar}(A'(n))]^2}{[\mathrm{Ar}(A(n))]^2}=\frac{A'_1(n)A'_2(n)A'_3(n)(A'_1(n)+A'_2(n)+A'_3(n))}{A_1(n)A_2(n)A_3(n)(A_1(n)+A_2(n)+A_3(n))}\\
&\leq\frac{A'_1(n)A'_2(n)A'_3(n)}{A_1(n)A_2(n)A_3(n)}\max_i\{\frac{A'_i(n)}{A_i(n)}\}.
\end{split}
\end{equation*}
Making $n\to \infty$, we get a contradiction.
\end{proof}

\begin{theorem}\label{the:complete}
	$(\frak{T}_1,\eta)$ is complete.
\end{theorem}
\begin{proof}
	By Theorem \ref{thm-crucial} it is enough to show that the metric $\eta_{max}$ is complete. This metric is the restriction of the metric defined on $(\R_+^*)^3$ given by
	$$d((A_1,A_2,A_3),(A'_1,A'_2,A'_3)\})=\max\{\lvert \log A'_i-\log A_i\rvert \}.$$

	Now consider the metric $d_1$ on $\R^3$ given by the following formula:
	$$d_1((B_1,B_2,B_3),(B'_1,B'_2,B'_3))=\max\lvert B'_i-B_i\rvert.$$
	It is not difficult to see that 
	$$E: \{B_{ijk}\}\to \{e^{B_{ijk}}\}$$
	\noindent is an isometry between $(\R^3,d_1)$ and $((\R_+^*)^3,d)$. The former metric space is complete, and $d_1$ is equivalent to the usual Euclidean metric. Now since $E^{-1}(\frak{T}_1)$ is a closed subspace of $\R^3$, the result follows.
\end{proof}

The metric $\eta$ is the restriction to the submanifold $\frak{T}_1$ of the {\it asymmetric Thompson distance} defined on the standard positive cone $(\R_+^*)^3$. The asymmetric Thompson distance on a cone in a finite-dimensional real vector spaces was introduced in \cite{SOP1}. Several properties of this distance have been studied, including its Finsler nature on suitable submanifolds and the description of its geodesics for the standard positive cone in $\R^n$.

\end{example}

\begin{example}
	\label{ex:quadrangle}
	Consider the space of unit-area quadrilaterals equipped with a triangulation.
	We assume that the triangulation is obtained by adding a diagonal  as in Figure \ref{fig:quadrilateral}.
	Let us denote this space by $\mathfrak{Q}_1$.
	
	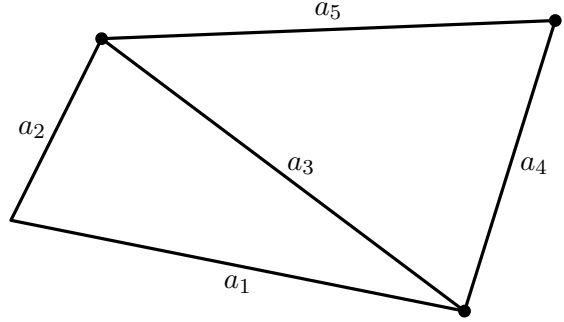
\begin{figure}[htbp]
		\centering
		\begin{tikzpicture}[scale=1.2, line cap=round, line join=round]
			
			% Köşeler
			\coordinate (A) at (0,3);
			\coordinate (B) at (5,3.2);
			\coordinate (C) at (4,0);
			\coordinate (D) at (-1,1);
			
			% Dörtgenin kenarları
			\draw[very thick] (A) -- (B) -- (C) -- (D) -- cycle;
			
			% İçteki doğru parçası
			\draw[very thick] (A) -- (C);
			
			% Noktalar
			\fill (A) circle (2pt);
			\fill (B) circle (2pt);
			\fill (C) circle (2pt);
			
			% Kenar etiketleri
			\node[above] at ($(A)!0.5!(B)$) {$a_5$};
			\node[right] at ($(B)!0.5!(C)$) {$a_4$};
			\node[below] at ($(D)!0.5!(C)$) {$a_1$};
			\node[left]  at ($(D)!0.5!(A)$) {$a_2$};
			
			% İç çizgi etiketi
			\node at ($(A)!0.5!(C)+(0.2,0.1)$) {$a_3$};
			
		\end{tikzpicture}
		\caption{A quadrilateral with a diagonal.}
		\label{fig:quadrilateral}
	\end{figure}
	
	Given such a quadrilateral, we introduce the following parameters:
	\[
	A_{123}= \frac{a_2+a_3-a_1}{2}, \qquad
	A_{213}= \frac{a_1+a_3-a_2}{2}, \qquad
	A_{312}= \frac{a_1+a_2-a_3}{2},
	\]
	and
	\[
	A_{345}= \frac{a_4+a_5-a_3}{2}, \qquad
	A_{435}= \frac{a_3+a_5-a_4}{2}, \qquad
	A_{534}= \frac{a_3+a_4-a_5}{2}.
	\]
	
	Let $Q$ and $Q'$ be triangulated quadrilaterals with associated parameters $(A_{ijk})$ and $(A'_{ijk})$, respectively. Then the function
	\[
	\eta:\mathfrak{Q}_1\times \mathfrak{Q}_1\to \mathbb{R}
	\]
	given by
	\[
	\eta(Q,Q')
	=
	\log\left(
	\max_{ijk}\frac{A'_{ijk}}{A_{ijk}}
	\right)
	\]
	is a metric on $\mathfrak{Q}_1$, see \cite{SOP1}. Now consider the two sequences of quadrilaterals given in Figure \ref{fig:right-triangle-decompositions}. We denote by $Q_n$ and $Q'_n$ the corresponding normalised unit-area quadrilaterals.
	
	\begin{figure}[htbp]
		\centering
		\begin{tikzpicture}[
			scale=1.9,
			line cap=round,
			line join=round,
			every node/.style={font=\small}
			]
			
			% ---------------- Left figure ----------------
			\begin{scope}
				\coordinate (A) at (0,1.414);
				\coordinate (B) at (0,0);
				\coordinate (C) at (-0.707,0.707);
				\coordinate (D) at (0.48,1.226);
				
				% Outer shape and diagonal
				\draw[thick] (C) -- (A) -- (D) -- (B) -- cycle;
				\draw[thick] (A) -- (B);
				
				% Right angle marks
				\pic[draw, thick, angle radius=2.7mm] {right angle = A--C--B};
				\pic[draw, thick, angle radius=2.7mm] {right angle = A--D--B};
				
				% Points
				\fill (A) circle (1.1pt);
				\fill (B) circle (1.1pt);
				\fill (C) circle (1.1pt);
				\fill (D) circle (1.1pt);
				
				% Labels
				\node[left] at ($(C)!0.5!(A)+(-0.04,0.04)$) {$1$};
				\node[left] at ($(C)!0.5!(B)+(-0.04,-0.04)$) {$1$};
				\node[left] at ($(A)!0.5!(B)$) {$\sqrt{2}$};
				
				\node[above] at ($(A)!0.5!(D)+(0,0.05)$) {$\dfrac{1}{n}$};
				\node[right] at ($(D)!0.5!(B)+(0.10,0)$)
				{$\sqrt{2-\dfrac{1}{n^2}}$};
			\end{scope}
			
			% ---------------- Right figure ----------------
			\begin{scope}[xshift=4.2cm]
				\coordinate (A) at (0,1.414);
				\coordinate (B) at (0,0);
				\coordinate (C) at (-0.707,0.707);
				\coordinate (D) at (0.65,0.985);
				
				% Outer shape and diagonal
				\draw[thick] (C) -- (A) -- (D) -- (B) -- cycle;
				\draw[thick] (A) -- (B);
				
				% Right angle marks
				\pic[draw, thick, angle radius=2.7mm] {right angle = A--C--B};
				\pic[draw, thick, angle radius=2.7mm] {right angle = A--D--B};
				
				% Points
				\fill (A) circle (1.1pt);
				\fill (B) circle (1.1pt);
				\fill (C) circle (1.1pt);
				\fill (D) circle (1.1pt);
				
				% Labels
				\node[left] at ($(C)!0.5!(A)+(-0.04,0.04)$) {$1$};
				\node[left] at ($(C)!0.5!(B)+(-0.04,-0.04)$) {$1$};
				\node[right] at ($(A)!0.5!(B)+(0.04,0)$) {$\sqrt{2}$};
				
				\node[above] at ($(A)!0.5!(D)+(0,0.05)$) {$\dfrac{2}{n}$};
				\node[right] at ($(D)!0.5!(B)+(0.12,0)$)
				{$\sqrt{2-\dfrac{4}{n^2}}$};
			\end{scope}
			
		\end{tikzpicture}
		\caption{}
		\label{fig:right-triangle-decompositions}
	\end{figure}
	
	It is not difficult to see that
	\[
	\eta(Q_n,Q'_n)\to \log 2
	\qquad\text{and}\qquad
	\eta(Q'_n,Q_n)\to 0
	\]
	as $n\to\infty$. Therefore, $\eta$ is not convergence-symmetric.
	 
\end{example}

\begin{example}
	Consider the space of marked convex unit-area quadrilaterals (without a fixed triangulation), where the marking refers to the vertices. We denote this space by $\mathfrak{CQ}_1$. Each element of this space admits two distinct triangulations, as shown in Figure \ref{fig:two-quadrangles-different-diagonals}.
	
	\begin{figure}[htbp]
		\centering
		\begin{tikzpicture}[
			scale=1.35,
			line cap=round,
			line join=round,
			every node/.style={font=\small}
			]
			
			% ==================================================
			% Left quadrangle
			% ==================================================
			\begin{scope}
				% Same quadrangle vertices
				\coordinate (A) at (0,1.2);    % left
				\coordinate (B) at (1.4,2.8);  % top
				\coordinate (C) at (3.6,2.3);  % right
				\coordinate (D) at (1.8,0);    % bottom
				
				% Outer quadrangle
				\draw[thick] (A) -- (B) -- (C) -- (D) -- cycle;
				
				% Diagonal
				\draw[thick] (B) -- (D);
				
				% Vertex dots
				\fill (A) circle (1.4pt);
				\fill (B) circle (1.4pt);
				\fill (C) circle (1.4pt);
				\fill (D) circle (1.4pt);
				
				% Edge labels
				\node[left]  at ($(A)!0.55!(B)+(-0.15,0.05)$) {$a_2$};
				\node[left]  at ($(A)!0.45!(D)+(-0.10,-0.12)$) {$a_1$};
				\node[above] at ($(B)!0.5!(C)+(0,0.10)$) {$a_5$};
				\node[right] at ($(C)!0.5!(D)+(0.12,0)$) {$a_4$};
				
				% Diagonal label
				\node at ($(B)!0.55!(D)+(0.22,0)$) {$a_3$};
			\end{scope}
			
			% ==================================================
			% Right quadrangle (same shape, different diagonal)
			% ==================================================
			\begin{scope}[xshift=6.5cm]
				% Same quadrangle vertices
				\coordinate (A) at (0,1.2);    % left
				\coordinate (B) at (1.4,2.8);  % top
				\coordinate (C) at (3.6,2.3);  % right
				\coordinate (D) at (1.8,0);    % bottom
				
				% Outer quadrangle
				\draw[thick] (A) -- (B) -- (C) -- (D) -- cycle;
				
				% Different diagonal
				\draw[thick] (A) -- (C);
				
				% Vertex dots
				\fill (A) circle (1.4pt);
				\fill (B) circle (1.4pt);
				\fill (C) circle (1.4pt);
				\fill (D) circle (1.4pt);
				
				% Edge labels
				\node[left]  at ($(A)!0.55!(B)+(-0.15,0.05)$) {$a_2$};
				\node[left]  at ($(A)!0.45!(D)+(-0.10,-0.12)$) {$a_1$};
				\node[above] at ($(B)!0.5!(C)+(0,0.10)$) {$a_5$};
				\node[right] at ($(C)!0.5!(D)+(0.12,0)$) {$a_4$};
				
				% Diagonal label
				\node[below] at ($(A)!0.5!(C)+(0,-0.10)$) {$a_3'$};
			\end{scope}
			
		\end{tikzpicture}
		\caption{Two identical quadrilaterals with different diagonals.}
		\label{fig:two-quadrangles-different-diagonals}
	\end{figure}
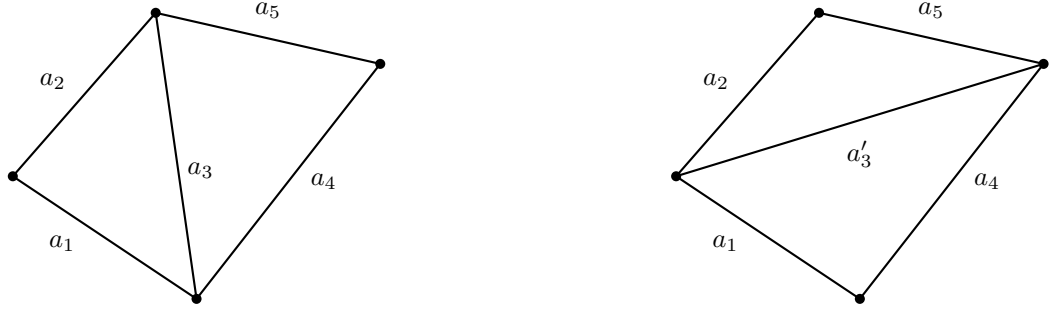
	
	Let
	$
	A_{123}, A_{213}, A_{312}, A_{345}, A_{435},$ and $A_{534}
	$
	be the quantities defined in the preceding example. We also define $\eta$ as in the preceding example. Thus, $\eta$ is a metric on $\mathfrak{CQ}_1$.
	
	We now introduce another set of parameters. Given a quadrilateral $Q$ in $\mathfrak{CQ}_1$, let
	\[
	\bar{A}_{143}
	=
	\frac{a_1+a_4-a_3'}{2},
	\qquad
	\bar{A}_{431}
	=
	\frac{a_4+a_3'-a_1}{2},
	\qquad
	\bar{A}_{314}
	=
	\frac{a_3'+a_1-a_4}{2},
	\]
	and
	\[
	\bar{A}_{352}
	=
	\frac{a_3'+a_5-a_2}{2},
	\qquad
	\bar{A}_{523}
	=
	\frac{a_5+a_2-a_3'}{2},
	\qquad
	\bar{A}_{235}
	=
	\frac{a_2+a_3'-a_5}{2}.
	\]
	
	Let $Q$ and $Q'$ be quadrilaterals in $\mathfrak{CQ}_1$ with associated parameters $(\bar{A}_{ijk})$ and $(\bar{A}'_{ijk})$ respectively. Define
	\[
	\bar{\eta}:\mathfrak{CQ}_1\times \mathfrak{CQ}_1\to \mathbb{R}
	\]
	by
	\[
	\bar{\eta}(Q,Q')
	=
	\log\left(
	\max_{ijk}
	\frac{\bar{A}'_{ijk}}{\bar{A}_{ijk}}
	\right),
	\]
	where the maximum is taken over
	\[
	ijk\in \{143,431,314,352,523,235\}.
	\]
	Then $\bar{\eta}$ is also a metric on $\mathfrak{CQ}_1$.
	
	Consider the sequences of quadrilaterals in Example \ref{ex:quadrangle}. Then, using Ptolemy's formula for quadrilaterals inscribed in a circle, it is not difficult to see that
	\[
	\bar{\eta}(Q_n,Q_n')\to \log 2
	\]
	as $n\to\infty$, whereas
	\[
	\bar{\eta}(Q_n',Q_n)\to 0
	\]
	as $n\to\infty$.
	
	Now let
	\[
	\eta_{\max}(Q,Q')
	=
	\max\{\eta(Q,Q'),\bar{\eta}(Q,Q')\}.
	\]
	Then $\eta_{\max}$ is a metric on $\mathfrak{CQ}_1$. The above discussion implies that $\eta_{\max}$ is not convergence-symmetric either.
\end{example}

It is known that Thurston's metric, in its original form, satisfies Busemann's axiom, and that it is forward and backward complete, see \cite{2007a}.
The following questions concern this metric.

\begin{question}[Convergence-symmetry of Thurston's metric]
Is Thurston's metric convergence-symmetric? 
\end{question}

\begin{question}[The symmetrisations of the Thurston metric]
Study the max and sup symmetrisations of Thurston's metric: decide whether or not they are Finsler and give descriptions of (at least a class of) their geodesics.
\end{question}

The same questions can be asked for Thurston's arc metric, that is, an analogue of Thurston's metric for surfaces with boundary (see \cite{PS2015, HP2021}).

After Thurston's metric and its variants, it is natural to mention the earthquake metric, which originates in the same paper of Thurston in which he introduced the Thurston metric \cite{thurston}. We shall recall the definition and  address  some questions related to this metric.

\begin{example}[The earthquake metric] Let  $S=S_{g,n}$ be a surface of finite type of genus $g\geq 0$ with $n\geq 0$ punctures and of negative Euler characteristic, and let $\mathcal{T}(S)$ be its Teichm\"uller space. A result of Kerckhoff from \cite[Proposition 2.6]{Kerckhoff-analytic} says that for every point $x$ in $\mathcal{T}(S)$, every vector $v$  in the tangent space $T_x(\mathcal{T}(S))$ of $\mathcal{T}(S)$ at $x$ can be uniquely expressed as the tangent vector to a right earthquake path. Such an earthquake path is associated with a measured lamination $\lambda$ on $S$. We write this tangent vector as ${\bf e}_{\lambda}(x)$ to show the dependence on $x$ and on $\lambda$.
In the paper quoted, Thurston showed that for any point $x\in \mathcal{T}(S)$,  the set of tangent vectors $v$ of the form ${\bf e}_{\lambda}(x)$ with $\ell_{\lambda}(x)=1$, where $\ell_{\lambda}(x)$ denotes the length of $\lambda$ measured in the hyperbolic structure $x$, is the boundary of a convex ball in $T_x(\mathcal{T}(S))$ containing the origin in its interior. The set of such convex balls, for $x$ varying in $\mathcal{T}(S)$, defines a Finsler structure on $\mathcal{T}(S)$. The associated metric is called the \emph{earthquake metric}.
\end{example}

\begin{question}[Convergence symmetry of the earthquake metric]
In the paper \cite{earthquake-metric}, the authors studied extensively the earthquake metric. They introduced a new notion of completion for asymmetric metrics, which they called FD-completion (Finite-Distance completion). They proved that the earthquake metric is FD-complete, that its FD-completion coincides with the completion of various symmetrisations of this metric, and that these completions coincide with the familiar Weil--Petersson completion of Teichm\"uller space. 
The question of whether or not the earthquake metric is convergence-symmetric is open.
\end{question}

We end this section with some questions concerning other metrics motivated by the geometry of surfaces. We start with a non-symmetric version of the well-known Hausdorff distance on the space of subsets of a symmetric metric space. 
\begin{definition}[Left Hausdorff distance] \label{def:left-Hausdorff}
Let $(X,d)$ is a symmetric metric space and let $\mathcal{C}(X)$ be the set of compact subsets of $X$.

The left Hausdorff distance on $\mathcal{C}(X)$ is defined, for $A$ and $B$ in  $\mathcal{C}(X)$, by the formula
 $$ d_{\vec{H}}(A,B)= \inf_{\epsilon >0}\{ A \subset N_{\epsilon}(B)\}.$$
 \end{definition}

 In the paper \cite{Ohshika2}, the authors studied the left Hausdorff distance on the space $\mathcal{GL}(S)$  of geodesic laminations on a 
closed hyperbolic surface $S$. A  
 bijection $f: \mathcal{GL}(S)\to \mathcal{GL}(S)$, is said to \emph{preserves left Hausdorff convergence} if for any sequence 
$(\lambda_i)$ in $\mathcal{GL}(S)$ and for any $\mu$  in $\mathcal{GL}(S)$ we have
$$d_{\vec{H}}(\lambda_i,\mu) \to 0\iff d_{\vec{H}}(f(\lambda_i),f(\mu)) \to 0.$$
The following is proved in \cite[Theorem 1.3]{Ohshika2}:
\begin{theorem} \label{th:rigidity-left-Hausdorff}
The natural homomorphism from
the extended mapping class group of $S$ into the group of bijections of the
space of geodesic laminations which preserve left Hausdorff convergence
 is an isomorphism.
\end{theorem}

Note that the left Hausdorff distance on the space $\mathcal{GL}(S)$  of geodesic laminations  is a weak metric which is not convergence-symmetric: Take two geodesic laminations $\lambda$ and $\mu$ such that $\lambda$ is a strict subset of $\lambda$. Then the two sequences $(x_n)$ and $(y_n)$ defined by  $x_n=\lambda$ and $y_n=\mu$ for all $n\geq 0$ satisfy 
 $d_{\vec{H}}(\lambda_n,\mu_n)=0$
 $d_{\vec{H}}(\mu_n,\lambda_n)>0$    for all $n$.

 The left Hausdorff distance is not proper to the setting of laminations; one may define it on various spaces of compact subsets of a given metric space. For instance, it can be defined on the set $\mathcal{C}$ of nonempty compact subsets of the Euclidean plane, of the hyperbolic plane, of a hyperbolic surface, etc., or on a special subset of such a set $\mathcal{C}$. (In the above example, it is defined on a subset of the set of compact subsets of a compact hyperbolic surface.) We propose the following
 \begin{question}
 Study the left Hausdorff distance in a simple case, e.g. on the set of nonempty compact subsets of the Euclidean plane, of the sphere $S^2$ or of the hyperbolic plane $\mathbb{H}^2$, and characterise its isometry group. We conjecture that one obtains in this way good examples of Finsler spaces, that is, one may characterise their geodesics, isometries, etc.
 \end{question}
 
%  \begin{question}
%  
% \end{question}

\end{document}